\documentclass[a4paper, 11pt]{article}

\usepackage[arxiv]{optional}
\usepackage[english]{babel}
\usepackage[T1]{fontenc}
\usepackage{amsmath}
\opt{arxiv}{%
\usepackage{amsthm}
}%
\usepackage{amsfonts}
\usepackage{amssymb}
\usepackage[final]{graphicx}
\usepackage{paralist}
\usepackage{subfigure}
\opt{draft}{%
\usepackage{refcheck}
}
\usepackage{url}

\opt{arxiv}{%
\theoremstyle{plain}
\newtheorem{thm}{Theorem}[section]
\newtheorem{prop}[thm]{Proposition}
\newtheorem*{prop*}{Proposition}
\newtheorem{lem}[thm]{Lemma}
\newtheorem*{lem*}{Lemma}
\newtheorem{cor}[thm]{Corollary}
\newtheorem*{cor*}{Corollary}
\newtheorem{example}[thm]{Example}
\newtheorem*{example*}{Example}
\newtheorem{conject}[thm]{Conjecture}
\newtheorem*{conject*}{Conjecture}

\theoremstyle{definition}
\newtheorem{defn}[thm]{Definition}
\newtheorem*{defn*}{Definition}
\newtheorem{rem}[thm]{Remark}
\newtheorem*{rem*}{Remark}
}%
\opt{springer}{%

}%

\newcommand{\eps}{\varepsilon}

\DeclareMathOperator{\ind}{ind}

\DeclareMathOperator{\new}{new}

\opt{springer}{\newcommand{\eop}{\qed}}
\opt{arxiv}{\newcommand{\eop}{}}

\newcommand{\abs}[1]{\vert #1 \vert}

\newcommand{\disk}{D}
\newcommand{\cdisk}{\overline{D}}
\newcommand{\annu}{A}
\newcommand{\cannu}{\overline{A}}

\newcommand{\R}{\ensuremath{\mathbb{R}}}
\newcommand{\C}{\ensuremath{\mathbb{C}}}

\newcommand{\Z}{\ensuremath{\mathbb{Z}}}

\newcommand{\coloneq}{\mathrel{\mathop:}=}

\newcommand{\conj}[1]{\overline{#1}}

\title{Perturbing rational harmonic functions by poles}
\author{Olivier S\`{e}te, Robert Luce, J\"{o}rg Liesen}

\begin{document}
\maketitle

\begin{abstract}
We study how adding certain poles to rational harmonic functions of the form
$R(z)-\conj{z}$, with $R(z)$ rational and of degree $d\geq 2$, affects the
number of zeros of the resulting functions. Our results are motivated
by and generalize a construction of Rhie derived in the context
of gravitational microlensing (arXiv:astro-ph/0305166). Of particular
interest is the construction and the behavior of rational functions $R(z)$
that are {\em extremal} in the sense that $R(z)-\conj{z}$ has the maximal
possible number of $5(d-1)$ zeros.

\end{abstract}

\section{Introduction}

This work is concerned with the zeros of functions in the complex variable 
$z$ of the form $R(z)-\conj{z}$,
where $R(z)$ is a rational function. The analysis of such rational harmonic
functions has received considerable attention in recent years. As nicely
explained in the expository article of Khavinson and
Neumann~\cite{KhavinsonNeumann2008},
they have important applications in gravitational microlensing; see also the
survey~\cite{PettersWerner2010}. In addition they are related to the
matrix theory problem of expressing certain adjoints of a diagonalizable
matrix as a rational function in the matrix~\cite{Liesen2007}.

In the sequel, whenever we write a rational function $R(z) = \tfrac{p(z)}{q(z)}$,
we assume that the polynomials $p(z)$ and $q(z)$ are coprime, i.e., that they
have no common zero. Then the \emph{degree} of $R(z)$, denoted by $\deg(R)$, is
defined as the maximum of the degrees of $p(z)$ and $q(z)$.

It is easy to see that $R(z)-\conj{z}$ has exactly one
zero if $\deg(R)=0$, and either $0,1,2$ or infinitely many zeros (given by all
points of a line or a circle in the complex plane) if $\deg(R)=1$.
More interesting is the case $\deg(R)=d\geq 2$, which we consider in this paper.
An important result of Khavinson and Neumann states that in this case
$R(z) - \conj{z}$ may have at most $5(d-1)$ zeros~\cite{KhavinsonNeumann:2006}.
Prior to the work of Khavinson and Neumann the astrophysicist
Sun Hong Rhie (1955--2013) had constructed, in the context of gravitational
lensing, examples of functions $R(z)-\conj{z}$ with exactly $5(d-1)$ zeros
for every $d\geq 2$~\cite{Rhie2003}. Hence the bound of Khavinson and Neumann
is sharp for every $d\geq 2$.

Motivated by the result of Khavinson and Neumann we call a rational function
$R(z)$ of degree $d\geq 2$ \emph{extremal}, when the function $R(z)-\conj{z}$
has the maximal possible number of $5(d-1)$ zeros.
Examples of such extremal rational functions in the published literature are rare.
For $d=2$ an example is given in \cite{KhavinsonNeumann:2006}.
The only other published example we are aware of is given by the construction
of Rhie~\cite{Rhie2003} and in the closely related
works~\cite{BayerDyer2007,BayerDyer2009}; see
also~\cite{LuceSeteLiesen2013}.
Rhie's construction served as a motivation for our work, and some of our results
can be considered a generalization of her original idea.

To briefly explain this idea, consider the rational harmonic function
\begin{equation*}
R_0(z)-\conj{z},\quad\text{where}\quad R_0(z) = \tfrac{z^{d-1}}{z^d - r^d}
\end{equation*}
and $r>0$ is a real parameter. For $d = 2$ a straightforward
computation shows that $R_0(z)-\conj{z}$ has $5$ zeros if $r < 1$, and
fewer zeros for $r\geq 1$.  Moreover, it was shown
in~\cite{MaoPettersWitt1999} that for $d\geq 3$ the function
$R_0(z)-\conj{z}$ has $3d+1$ zeros if $r <\big( \tfrac{d-2}{d}
\big)^{\frac{1}{2}} \big( \tfrac{2}{d-2}\big)^{\frac{1}{d}}$, and
fewer zeros for larger values of $r$. Thus, $R_0(z)$ is extremal only
for $d=2$ or $d=3$ (and when $r$ is small enough).  Rhie suggested
in~\cite{Rhie2003} to perturb $R_0(z)-\conj{z}$ by adding a pole at
one of its zeros, namely at the point $z=0$. More precisely, she
showed that for a particular value $r>0$ and sufficiently small
$\eps>0$ the rational harmonic function
\begin{equation}
\label{eqn:R_eps}
R_\eps(z)-\conj{z},\quad\text{where}\quad
R_\eps(z) = (1 - \eps) R_0(z) + \tfrac{\eps}{z},
\end{equation}
has $5d$ zeros, so that $R_\eps(z)$ (of degree $d+1$) is indeed extremal.
An elementary proof and sharp bounds on the parameters $r$ and $\eps$ that guarantee extremality of
$R_\eps(z)$ are given in~\cite{LuceSeteLiesen2013}.

\begin{figure}[t]
\subfigure[Phase portrait of $R_0(z) - \conj{z}$]{%
\label{fig:circlens1}%
\includegraphics[width=.49\textwidth]{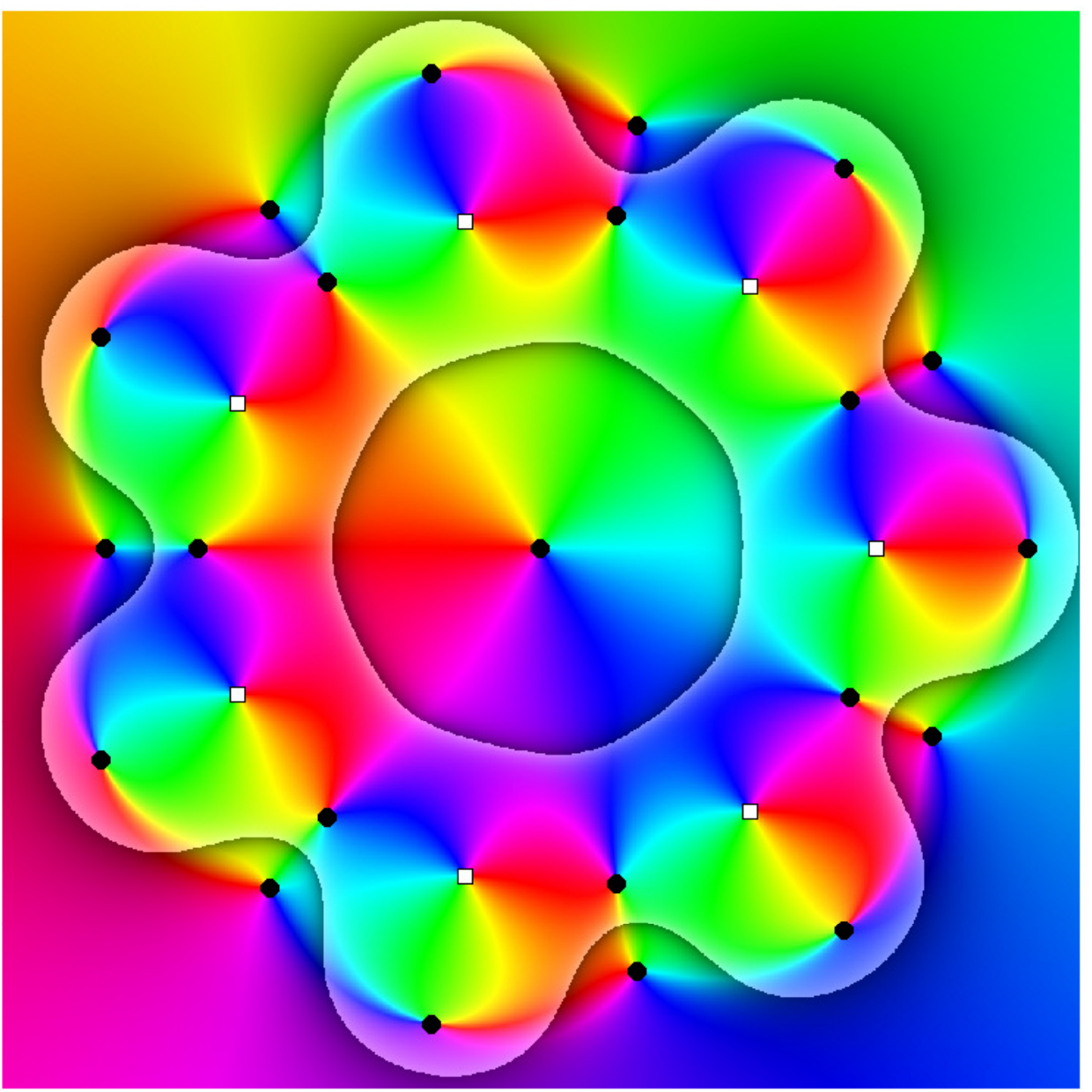}
}
\hfill
\subfigure[Phase portrait of $R_\eps(z) - \conj{z}$]{%
\label{fig:circlens2}%
\includegraphics[width=.49\textwidth]{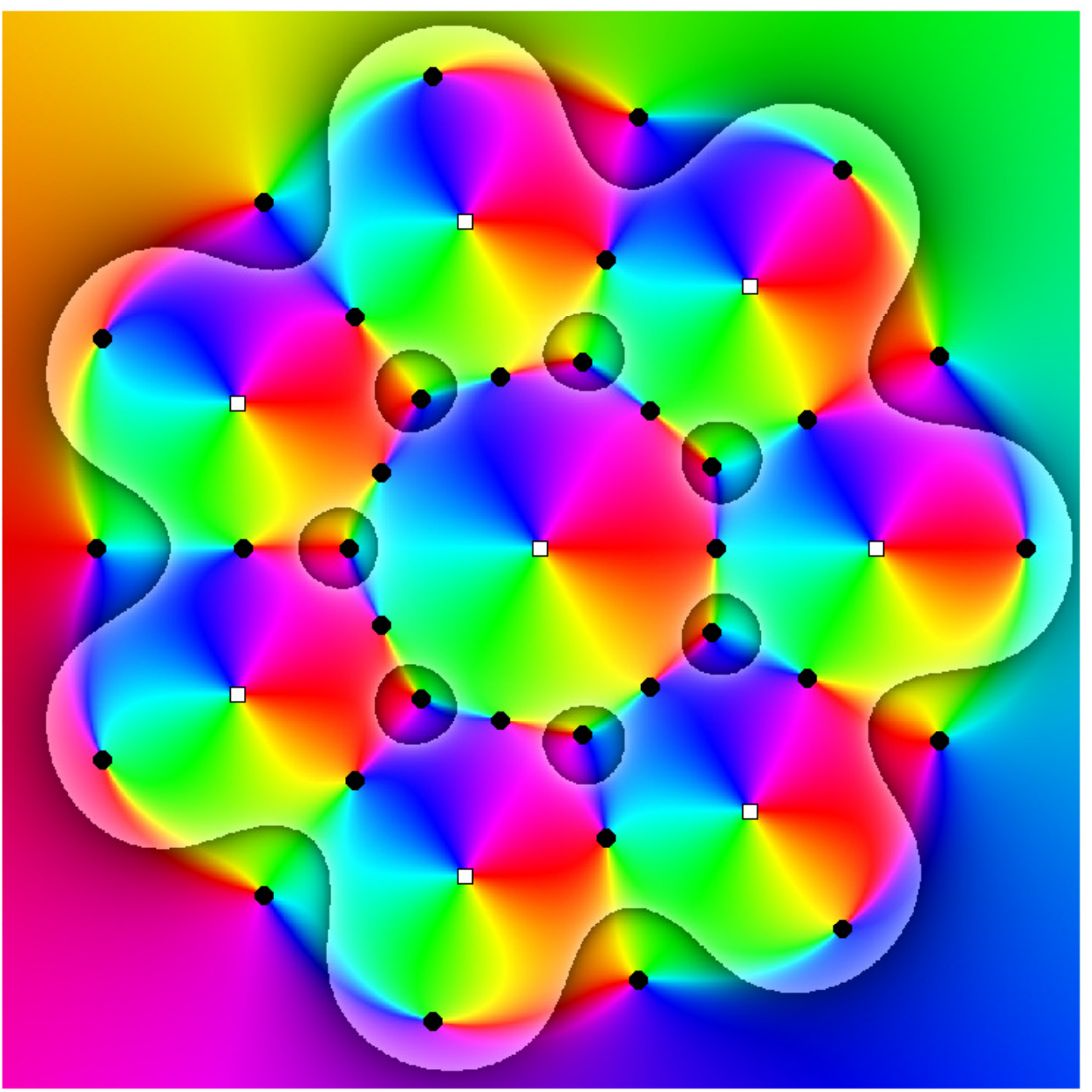} }
\caption{Phase portraits of $R_0(z) - \conj{z}$ and $R_\eps(z) - \conj{z}$.
Black disks indicate zeros and white squares show
poles.\label{fig:circlens}}
\end{figure}

A numerical example for Rhie's construction is shown in
Figure~\ref{fig:circlens}, where we use phase portraits for
visualization~\cite{WegertSemmler2011,Wegert2012}; see also
Section~\ref{sec:examples} below. Here $d=\deg(R_0)=7$, and $r>0$ is
chosen sufficiently small. Figure~\ref{fig:circlens1} shows that
$R_0(z) - \conj{z}$ has $3d+1=22$ zeros and $7$ poles. We choose
$\eps=0.15$, which is sufficiently small in order to obtain an
extremal function $R_\eps(z)$, i.e. a function $R_\eps(z)-\conj{z}$
with $5d=35$ zeros; see Figure~\ref{fig:circlens2}.  Two essential
observations are to be made when comparing Figures~\ref{fig:circlens1}
and~\ref{fig:circlens2}:
\begin{itemize}

\item[(1)] The zero of $R_0(z)-\conj{z}$ at the perturbation point
$z=0$ becomes a pole of the perturbed function, but all other zeros of
$R_0(z)-\conj{z}$ ``survive'' the perturbation in the sense that
perturbed versions of them are still zeros of $R_\eps(z)-\conj{z}$.
\item[(2)] In a neighborhood of the perturbation point $z=0$ the
function $R_\eps(z)-\conj{z}$ has $2d$ additional zeros, which are
located on two circles of radius approximately $\sqrt{\eps}$ around
$z=0$. (In Figure~\ref{fig:circlens2} the two circles are
visually almost indistinguishable.)

\end{itemize}

While the function $R_\eps(z)$ considered by Rhie arises from a convex
combination of $R_0(z)$ with the pole $\tfrac{1}{z}$, it was shown
in~\cite{LuceSeteLiesen2013} that qualitatively the same ``zero
creating effect'' can be observed for a purely additive perturbation
with $\frac{1}{z}$, resulting in the function
$R_0(z)+\tfrac{\eps}{z}-\conj{z}$. The analyses of this effect draw
heavily from the rotational symmetry of $R_0(z)$. Hence it is tempting
to regard this effect as a special property of $R_0(z)$.

The main goal of this paper is to prove, however, that \emph{the
effect is generic to any rational harmonic function} $R(z)-\conj{z}$,
provided that it satisfies certain conditions at the perturbation
point.  In more details, we show that if $z_0\in \mathbb{C}$ is a zero
of $R(z)-\conj{z}$ and if $n-1$ is the order of the first
non-vanishing derivative of $R(z)$ at $z_0$, then $R(z)-\conj{z}$ can
be perturbed by a pole at $z_0$ such that the perturbed rational
harmonic function has (at least) $2n$ additional zeros in a
neighborhood of $z_0$, while the other zeros of $R(z)-\conj{z}$
``survive'' the perturbation. The precise statement of this result is
given in Theorem~\ref{thm:pert_new_zeros} below.  In order to fully
characterize this zero creating effect, perturbations with poles at
arbitrary points $z_0 \in \C$ are studied in
Theorem~\ref{thm:other_z0}.

Adding poles to rational harmonic functions of the form $R(z) -
\conj{z}$ in order to create new zeros has been used, for example,
in~\cite{BayerDyer2007,BleherEtAl2014}.  However, a complete
mathematical characterization of this effect as in this work has not
been given before.

We would also like to point out that the conceptually similar problem
of constructing extremal harmonic \emph{polynomials} $p(z) - \conj{z} = 0$
appears to be more challenging.  A class of examples that realizes the
maximal number of zeros $3\deg(p) - 2$ (see~\cite{KhavinsonSwiatek2003}) is 
given by Geyer~\cite{Geyer2008}.  Recent progress concerning Wilmshurst's
conjecture~\cite{Wilmshurst1998} on the maximal number of zeros for
general harmonic polynomials $p(z) - \conj{q(z)}$ has been reported
in~\cite{LeeLerarioLundberg2013}.

The paper is organized as follows. In Section~\ref{sec:background} we review
tools from harmonic function theory that we need in our proofs. The major part of
Section~\ref{sec:perturb} consists of the proof of Theorem~\ref{thm:pert_new_zeros}. 
We first show that sufficiently many zeros are created in a neighborhood of
the 
perturbation point (Section~\ref{sec:near_z0}), and we then show that the effect of 
the perturbation is local in the sense that the remaining zeros ``survive'' the
perturbation (Section~\ref{sec:away_z0}). In Section~\ref{sec:examples} we give 
several numerical illustrations. In Section~\ref{sec:conclusions} we state conclusions 
and some open questions in the context of our result.


\section{Mathematical background}
\label{sec:background}

In this section we review the required mathematical background
including the winding of continuous functions, the Poincar\'e index of
exceptional points and an argument principle that will allow to
``count'' zeros of $R(z) - \conj{z}$ in Section~\ref{sec:perturb}.

The functions of the form $f(z) = R(z) - \conj{z}$ we consider in this paper
are obviously not analytic. They are \emph{harmonic} because for $z=x+iy$ they are
twice continuously differentiable with respect to the real variables $x$ and $y$,
and additionally
$\tfrac{\partial^2 f}{\partial x^2} + \tfrac{\partial^2 f}{\partial y^2} = 0$.
In contrast to \cite{Duren2004} we make no assumption whether a harmonic function 
is (locally) bijective or not. Since $f$ is harmonic, it
has locally a representation of the form $f = h + \conj{g}$, where $h$
and $g$ are analytic functions. Both $h$ and $g$ are unique up to additive
constants; see~\cite{DurenHengartnerLaugesen1996}.

Roughly speaking, a harmonic function $h + \conj{g}$ is called
sense-preserving if the analytic part $h$ is dominant, and sense-reversing if
the co-analytic part $\conj{g}$ dominates.  Since the exact definition
simplifies for the harmonic functions of our interest, $f(z) = R(z) - \conj{z}$,
we give the definition for this case only and refer to~\cite{DurenHengartnerLaugesen1996}
and~\cite{SuffridgeThompson2000} for the general case.

\begin{defn} \label{defn:sense-pres-rev}
Let $f(z) = R(z) - \conj{z}$ and $z_0\in \C$. 
\begin{compactenum}
\item If $\abs{R'(z_0)} > 1$, then $f(z)$ is called \emph{sense-preserving} at $z_0$.
\item If $\abs{R'(z_0)} < 1$, then $f(z)$ is called \emph{sense-reversing} at $z_0$.
\item If $\abs{R'(z_0)} = 1$, then $z_0$ is called a \emph{singular point of $f(z)$}.
\end{compactenum}
In either case, if additionally $f(z_0) = 0$, then $z_0$ is respectively 
called a \emph{sense-preserving}, \emph{sense-reversing} or \emph{singular zero} of $f(z)$.
A zero $z_0$ of $f(z)$ that is not singular is called \emph{regular}.
If all zeros of $f(z)$ are regular, then the functions $f(z)$ and $R(z)$
are called \emph{regular}.
\end{defn}

\medskip

We now turn our attention to the zeros of harmonic functions and their
``multiplicity''.  It is clear that zeros of harmonic functions can,
in general, not be ``factored out''.  For example, each $z_0 \in \R$
is a zero of the harmonic function $z-\conj{z}$, but there exists no
harmonic function $g(z)$ with $z-\conj{z}=(z-z_0)g(z)$ or
$z-\conj{z}=\conj{(z-z_0)}g(z)$.  Even if all the zeros are isolated,
as it is the case for harmonic polynomials of the form $p(z) - \conj{z}$
with $\deg(p) > 1$, such decompositions typically do not exist, as the
number of zeros may exceed the degree of the
polynomial~\cite{Geyer2008,KhavinsonSwiatek2003}.  In order to still
``count'' zeros, we will use the concept  of the Poincar\'e index,
which will be defined below; see Definition~\ref{def:Poincare}.

In order to define the Poincar\'e index, we need
the following definition of the winding of a continuous
function~\cite{Balk1991}; see also \cite[p.~101]{Wegert2012} and
\cite[p.~29]{Sheil-Small2002} (where the winding is called ``degree'').
Let $\Gamma$ be a rectifiable curve with
parametrization $\gamma : [a, b] \to \Gamma$.  Let $f : \Gamma \to \C$
be a continuous function with no zeros on $\Gamma$.  Let $\Theta(z)$
denote a continuous branch of $\arg f(z)$ on $\Gamma$.  The
\emph{winding} (or \emph{rotation}) \emph{of $f(z)$ on the curve
$\Gamma$} is defined as
\begin{equation*}
V(f; \Gamma)
\coloneq \tfrac{1}{2 \pi} ( \Theta(\gamma(b)) - \Theta(\gamma(a)) )
= \tfrac{1}{2 \pi} \Delta_\Gamma \arg f(z).
\end{equation*}
The winding is independent of the choice of the branch of $\arg f(z)$.
The next proposition collects elementary properties of the winding.

\begin{prop}[{see \cite[p.~37]{Balk1991} or~\cite[p.~29]{Sheil-Small2002}}]
\label{prop:properties_winding}
Let $\Gamma$ be a rectifiable curve, and let $f(z)$ and $g(z)$ be continuous
and nonzero functions on $\Gamma$.
\begin{enumerate}
\item If $\Gamma$ is a closed curve, then $V(f; \Gamma)$ is an integer.

\item If $\Gamma$ is a closed curve and if there exists a continuous and
single-valued branch of the argument on $f(\Gamma)$, then $V(f; \Gamma) = 0$.


\item We have $V(f g; \Gamma) = V(f; \Gamma) + V(g; \Gamma)$.

\item If $f(z) = c \neq 0$ is constant on $\Gamma$, then $V(f; \Gamma) = 0$.

\end{enumerate}
\end{prop}

The next result is a version of Rouch\'e's theorem that is somewhat stronger
than the classical one; see~\cite[p.~37]{Balk1991}.  Its formulation
for analytic functions is due to Glicksberg~\cite{Glicksberg1976}; see
also~\cite{Conway1978}. We give a short proof for our more general
setting.

\begin{thm}[Rouch\'e's theorem] \label{thm:WindingRouche}
Let $\Gamma$ be a rectifiable closed curve, and let $f(z)$ and $g(z)$ be two
continuous
functions on $\Gamma$.  If
\begin{equation}
\label{eqn:WindingRouche:boundary_estimate}
\abs{ f(z) + g(z) } < \abs{ f(z) } + \abs{ g(z) }, \quad z \in \Gamma,
\end{equation}
then $f(z)$ and $g(z)$ have the same winding
on $\Gamma$, i.e. $V(f; \Gamma) = V(g; \Gamma)$.
\end{thm}
\begin{proof}
Note first that~\eqref{eqn:WindingRouche:boundary_estimate} implies that $f(z)$
and $g(z)$ are nonzero on $\Gamma$.  For $z \in \Gamma$, the
inequality~\eqref{eqn:WindingRouche:boundary_estimate} yields $\abs{ \tfrac{
f(z) }{ g(z) } + 1 } < \abs{ \tfrac{ f(z) }{ g(z) } } + 1$, hence
$\tfrac{ f(z) }{ g(z) } \in \C \backslash [0, \infty[$.  Thus $V(\tfrac{f}{g};
\Gamma) = 0$ since there is a continuous single-valued argument function on
$\C \backslash [0, \infty[$.  Using Proposition~\ref{prop:properties_winding}
we find $V(f; \Gamma) = V( g \tfrac{f}{g}; \Gamma) = V(g; \Gamma) +
V(\tfrac{f}{g}; \Gamma) = V(g; \Gamma)$.
\eop
\end{proof}

\begin{defn}\label{def:Poincare}
Let the function $f(z)$ be continuous and different from zero in a
punctured neighborhood $D$ of the point $z_0$.  If $f(z)$ is either zero,
not continuous, or not defined at $z_0$, then the point $z_0$ is
called an \emph{isolated exceptional point} of $f(z)$.  The
\emph{Poincar\'{e} index} of the (isolated) exceptional point $z_0$ of
the function $f(z)$ is defined as
\begin{equation*}
\ind(z_0; f) \coloneq V(f; \gamma) = \tfrac{1}{2 \pi} \Delta_\gamma \arg f(z)
\in \Z,
\end{equation*}
where $\gamma$ is an arbitrary closed Jordan curve in $D$ surrounding $z_0$.
\end{defn}

\medskip
The Poincar\'e index is independent of the choice of $\gamma$;
see~\cite{Balk1991}, where exceptional points are called
\emph{critical points}.  In~\cite[p.~44]{Sheil-Small2002} the Poincar\'e 
index is called the ``multiplicity of $f$ at $z_0$''.
It is a generalization of the multiplicity of a zero and
pole of a meromorphic function, as is shown in the next example.

\begin{example} \label{example:poincare_index_for_analytic_functions}
Let $f(z)$ be a meromorphic function.  The only isolated exceptional points of
$f(z)$ are its zeros and poles.  Suppose $f(z) = (z-z_0)^m g(z)$ holds in some
neighborhood of $z_0$, where $g(z)$ is analytic and nonzero in this
neighborhood, and $m \in \Z$.  Then, for a sufficiently small circle $\gamma$
around $z_0$ lying in this neighborhood, Proposition~\ref{prop:properties_winding} yields
\begin{equation*}
\ind(z_0; f) = V(f; \gamma) = V( (z-z_0)^m g(z); \gamma)
= m V(z-z_0; \gamma) + V(g; \gamma) = m.
\end{equation*}
Thus, for a zero of $f(z)$ ($m > 0$) the Poincar\'e index is the multiplicity of
the zero, and for a pole ($m < 0$) it is (minus) the order of the pole.
\end{example}

With the Poincar\'{e} index, an argument principle for continuous
functions with a finite number of exceptional points can be proven;
see~\cite[p.~39]{Balk1991} for the version below,
or~\cite[p.~44]{Sheil-Small2002}.
\begin{thm} \label{thm:arg_principle}
Let $\Gamma$ be a closed Jordan curve.  Let $f(z)$ be continuous on and
interior to $\Gamma$, except possibly for finitely many exceptional points $z_1,
z_2, \ldots, z_k$ in the interior of $\Gamma$.  Then
\begin{equation*}
\tfrac{1}{2 \pi} \Delta_\Gamma \arg f(z) = V(f; \Gamma)
= \ind(z_1;f) + \ind(z_2;f) + \ldots + \ind(z_k; f).
\end{equation*}
\end{thm}

This ``abstract argument principle'' includes as special cases the
argument principle for meromorphic functions (see
Example~\ref{example:poincare_index_for_analytic_functions}), and the
argument principles for harmonic functions
\cite{DurenHengartnerLaugesen1996} and for harmonic functions with
poles \cite[Theorems~2.2 and 2.3]{SuffridgeThompson2000}.

We end this section with a discussion of the exceptional points and their
Poincar\'{e} indices of $f(z) = R(z) - \conj{z}$, where $\deg(R) \geq 2$.  
The exceptional points of $f(z)$ are its zeros and its poles.  
The function $f(z)$ has a pole of order $m$ at $z_0$, if $z_0$ is a pole of 
order $m$ of $R(z)$. The assumption $\deg(R) \geq 2$ implies that
$f(z)$ has at most $5(\deg(R)-1)$ zeros (cf. the Introduction). 
Hence all exceptional points of $f(z)$ are isolated and thus have
a Poincar\'{e} index. 

The Poincar\'{e} index of a regular zero of a general harmonic
function can be read off the power series of its analytic and
co-analytic parts.  This characterization was implicitly obtained
in~\cite{DurenHengartnerLaugesen1996}.  Under certain conditions the
Poincar\'{e} index of a pole can be determined in a similar
way~\cite{SuffridgeThompson2000}.
Using~\cite[p.~412]{DurenHengartnerLaugesen1996}, Lemma~2.2 and the
argument principle in~\cite{SuffridgeThompson2000}, we obtain the
following characterization for rational harmonic functions of the form
$f(z) = R(z) - \conj{z}$.

\begin{prop} \label{prop:P_index_for_zeros_poles}
Let $f(z) = R(z) - \conj{z}$ with $\deg(R) \geq 2$.
\begin{compactenum}

\item A sense-preserving zero of $f(z)$ has Poincar\'{e} index $+1$ and a
sense-reversing zero of $f(z)$ has Poincar\'{e} index $-1$.

\item If $z_0$ is a pole of $R(z)$ of order $m$, then $f(z)$ is
sense-preserving in a neighborhood of $z_0$ and has Poincar\'{e} index
$-m$ at $z_0$.

\end{compactenum}
\end{prop}


\section{Creating zeros by adding poles}
\label{sec:perturb}

In Theorem~\ref{thm:pert_new_zeros}, the main result of this paper, we
generalize Rhie's construction as outlined in the Introduction.  In
short, the effect of an additive perturbation with a pole at a point
$z_0$ where $R(z) - \conj{z}$ has a sense-reversing zero at which some
derivatives of $R(z)$ vanish, can be roughly summarized as follows:
New zeros appear ``close to $z_0$'', and existing zeros of $R(z) -
\conj{z}$ ``away from $z_0$'' do not disappear.

We denote by $\disk(z,r)$ and $\cdisk(z,r)$ the open and the closed
disks around $z\in\C$ with radius $r>0$, respectively. The open and
closed annuli of radii $r>0$ and $s>0$ around $z\in\C$ are denoted by
$\annu(z, r, s)$ and $\cannu(z, r, s)$, respectively.

\begin{thm}
\label{thm:pert_new_zeros}
Let $f(z) = R(z) - \conj{z}$ with $\deg(R) \geq 2$ satisfy $\lim_{z \to \infty} f(z) = \infty$.
Further, let $z_0 \in \C$ and the integer $n\geq 3$ satisfy
\begin{equation}\label{eqn:z0}
f(z_0)=0,\quad R'(z_0) = \ldots = R^{(n-2)}(z_0) = 0,\quad\text{and}\quad
R^{(n-1)}(z_0)\neq 0,
\end{equation}
and set $\eta \coloneq (\tfrac{n}{n-1})^{\frac{1}{2}}$.
Then, for sufficiently small $\eps > 0$, the disk $\overline{D}(z_0, \eta\sqrt{\eps})$
contains no further zero of $f(z)$ and the function
\begin{equation*}
F(z) \coloneq f(z) + \tfrac{\eps}{z-z_0}
\end{equation*}
satisfies the following:
\begin{compactenum}
\renewcommand{\labelenumi}{(\roman{enumi})}
\item $F(z)$ has at least $n$ zeros in $A(z_0, \eta^{-1} \sqrt{\eps},
\sqrt{\eps})$, at least $n$ zeros in $A(z_0, \sqrt{\eps}, \eta
\sqrt{\eps})$, and no zeros in $\overline{D}(z_0, \eta^{-1} \sqrt{\eps})$.
If $F(z)$ is regular, then at least $n$ of its zeros in $A(z_0, \eta^{-1} \sqrt{\eps},
\sqrt{\eps})$ are sense-preserving and at least $n$ of its zeros in 
$A(z_0, \sqrt{\eps}, \eta \sqrt{\eps})$ are sense-reversing.
\item Denote the regular zeros of $f(z)$ outside $\cdisk(z_0, \eta
\sqrt{\eps})$ by $z_1, \ldots, z_N$.  Then there exist mutually disjoint disks
$\cdisk(z_k,r)$ such that $F(z)$ has exactly one zero
in each $\disk(z_k, r)$, and the index of this zero is $\ind(z_k; f)$.
\item If $f(z)$ is regular, then $f(z)$ and $F(z)$ have the same number of zeros
outside $\cdisk(z_0, \eta \sqrt{\eps})$.
\end{compactenum}
\end{thm}

\medskip

The proof of this result will be given in two parts, spanning
the Sections~\ref{sec:near_z0} and~\ref{sec:away_z0}.  First we will
give several remarks on the technical conditions on $f(z)$ and $F(z)$
we impose in this theorem.

\begin{rem}
\begin{compactenum}
\item The condition that $\lim_{z \to \infty} f(z) = \infty$ is rather
nonrestrictive.  First note that if $f(z)$ has a limit for $z \rightarrow
\infty$, this limit is $\infty$.  Further, $f(z)$ has no limit for $z \to \infty$ only when $R(z)
= a_1 z + \tfrac{p(z)}{q(z)}$, with $\abs{a_1} = 1$ and $\deg(p) \leq \deg(q)$.
For all other $R(z)$, the limit  $\lim_{z \to \infty} f(z)$ exists.

\item As shown in~\cite[p.~1081]{KhavinsonNeumann:2006},
the set of regular rational functions is dense in the space of all rational
functions with respect to the supremum norm on the Riemann sphere.  Hence,
restricting $f(z)$ to be regular in~\textit{(iii)} is a very mild condition.

\item However, the effects of a perturbation on a non-regular function $f(z)$
can be quite unpleasant: Zeros of $f(z)$ far away from $z_0$ can vanish or 
additional zeros may appear.  This is why the regularity assumptions 
in \textit{(ii)} and \textit{(iii)} are necessary.
\end{compactenum}
\end{rem}

\medskip

In particular, Theorem~\ref{thm:pert_new_zeros} can be used to explain
the maximality of Rhie's construction for a gravitational point lens,
which we briefly discussed in the introduction.  This is the content
of the following Corollary; see also
Section~\ref{sec:example_circlens} and Figure~\ref{fig:circlens}.

\begin{cor} \label{cor:produce_extremal_function}
Under the assumptions of Theorem~\ref{thm:pert_new_zeros}, suppose that
$f(z)=R(z)-\conj{z}$ is regular and has $3n+1$ zeros, and that $\deg(R) = n$.
Then, for sufficiently small $\eps > 0$, the function $R(z) +
\tfrac{\eps}{z-z_0}$ is extremal.
\end{cor}
\begin{proof}
By Theorem~\ref{thm:pert_new_zeros} the function $F(z)=R(z) + \tfrac{\eps}{z-z_0}-\conj{z}$ 
with sufficiently small $\eps>0$ has at least $2n$ zeros
inside  $A(z_0, \eta^{-1} \sqrt{\eps},\eta\sqrt{\eps})$ (cf. item~\textit{(i)})
and exactly $3n$ zeros outside $\overline{D}(z_0, \eta \sqrt{\eps})$ (cf.
item~\textit{(iii)}). Since $F(z)$ may have at most $5n$ zeros, it in fact has
exactly $5n$ zeros, so that $R(z) + \tfrac{\eps}{z-z_0}$ is extremal.
\eop
\end{proof}

In the following Section~\ref{sec:near_z0} we prove part \textit{(i)}
of Theorem~\ref{thm:pert_new_zeros}.  The proof of parts \textit{(ii)}
and \textit{(iii)} are the content of Section~\ref{sec:away_z0}.
Finally, in Section~\ref{sec:other_z0}, we analyze the effect of a
perturbation with a pole at points $z_0$ that do not satisfy the
assumptions of Theorem~\ref{thm:pert_new_zeros}.

\subsection{Behavior near $z_0$ (Proof of \textit{(i)} in Theorem~\ref{thm:pert_new_zeros})}
\label{sec:near_z0}

The idea of our proof is the following: We approximate $F(z)$ by a truncation of the 
Laurent series of the analytic part of $F(z)$, and we show that this approximation 
has $2n$ zeros near $z_0$ that carry over to zeros of $F(z)$, provided $\eps$ is sufficiently small.

Let
\begin{equation*}
R(z) = \sum_{k=0}^\infty \tfrac{R^{(k)}(z_0)}{k!} (z-z_0)^k
\end{equation*}
be the series expansion of $R(z)$ at $z_0$ (convergent in some suitable disk 
around $z_0$). For the moment, consider any $\eps>0$. We will specify below 
when $\eps$ is ``sufficiently small''. Using~\eqref{eqn:z0} we can write
\begin{equation*}
\begin{split}
F(z)&=R(z)+\tfrac{\eps}{z-z_0}-\conj{z}\\
&=\tfrac{R^{(n-1)}(z_0)}{(n-1)!}(z-z_0)^{n-1}+\sum_{k=n}^\infty \tfrac{
R^{(k)}(z_0) }{ k! } 
(z-z_0)^k+\tfrac{\eps}{z-z_0}-\conj{z-z_0}.
\end{split}
\end{equation*}
We substitute $w \coloneq z-z_0$, so that $z_0$ is mapped to $0$. To simplify
notation, we write again $F(w)$ instead of introducing a new notation
$\widetilde{F}(w) \coloneq F(z)$.  In the following we assume that the same
substitution has been applied to $R(z)$ and $f(z)$ in order to obtain $R(w)$ and
$f(w)$.
We set $c \coloneq \tfrac{R^{(n-1)}(0)}{ (n-1)! } \neq 0$ and hence obtain
\begin{equation}
    \label{eqn:F_apart}
    F(w) = c w^{n-1} + \sum_{k=n}^\infty \tfrac{ R^{(k)}(0) }{ k! } w^k
      + \tfrac{\eps}{w} - \overline{w}.
\end{equation}

Consider the function
\begin{equation}
G(w) \coloneq c w^{n-1} + \tfrac{\eps}{w} - \overline{w}, \label{eqn:defn_G}
\end{equation}
obtained from $F(w)$ by truncation of the series expansion of $R(w)$.
The zeros of $G(w)$ are the solutions of the equation
\begin{equation}
    \label{eqn:simplified_local_eqn}
    c w^n - \abs{w}^2 + \eps = 0.
\end{equation}
In the following, we will assume without loss of generality that $c > 0$.
Indeed, the solutions of~\eqref{eqn:simplified_local_eqn} with general $c \neq
0$ have the form $e^{- i \frac{\arg(c)}{n}} w$, where $w$ solves $\abs{c} w^n -
\abs{w}^2 + \eps = 0$.

The next three lemmata characterize the zeros of $G(w)$.  First, we will
derive conditions on $\eps$ such that $G(w)$ admits a maximal number
of zeros (Lemma~\ref{lem:num_zeros_G}).  Then, in
Lemma~\ref{lem:loc_zeros_G}, we derive bounds on the moduli of certain
such zeros that will be needed later.  The sense of these zeros is
determined in Lemma~\ref{lem:sense_zeros_G}.

\begin{lem}
\label{lem:num_zeros_G}
Let $n \geq 3$, $c > 0$ and $0 < \eps < \eps_\ast \coloneq
\tfrac{n-2}{n} \big( \tfrac{2}{nc} \big)^{\frac{2}{n-2}}$.
Then~\eqref{eqn:simplified_local_eqn} has exactly $3n$ solutions
$\rho_1 e^{i \frac{(2k+1) \pi}{n}}$, $\rho_2 e^{i \frac{2k \pi}{n}}$,
$\rho_3 e^{i \frac{2k \pi}{n}}$, $1 \leq k \leq n$, where $\rho_1 <
\sqrt{\eps} < \rho_2 < \big( \tfrac{2}{nc} \big)^{\frac{1}{n-2}} < \rho_3$.
\end{lem}
\begin{proof}
Write $w = \rho e^{i \varphi}$ with $\rho > 0$ and $\varphi \in \R$.
Equation~\eqref{eqn:simplified_local_eqn} is equivalent to
\begin{equation}
    \label{eqn:reduced_eqn_real_eps}
    c \rho^n e^{i n \varphi} - \rho^2 + \eps = 0,\quad
    \text{or}\quad e^{i n \varphi} = \tfrac{ \rho^2 - \eps }{ c \rho^n }.
\end{equation}
Thus $e^{i n \varphi}$ is real and we distinguish the two cases $e^{i n \varphi} =
\pm 1$.

If $e^{i n \varphi} = -1$, then $\rho^2 < \eps$, $\varphi = \tfrac{(2 k + 1)
\pi}{n}$ for some $k \in \Z$, and equation~\eqref{eqn:reduced_eqn_real_eps}
becomes $c \rho^n + \rho^2 - \eps = 0$.
By Descartes' rule of signs (see \cite[p.~442]{Henrici1974}), this equation has
exactly one positive root, say $\rho_1 < \sqrt{\eps}$.

If $e^{i n \varphi} = +1$, then $\rho^2 > \eps$,
$\varphi = \tfrac{2 k \pi}{n}$ for some $k \in \Z$,
and~\eqref{eqn:reduced_eqn_real_eps} yields
\begin{equation*}
f_+(\rho) \coloneq c \rho^n - \rho^2 + \eps = 0.
\end{equation*}
By Descartes' rule of signs $f_+(\rho)$ has $0$ or $2$ positive roots,
counting multiplicities.  We derive a necessary and sufficient condition on
$\eps$ such that $f_+(\rho)$ has two (distinct) positive roots.
Since $f_+'(\rho) = n c \rho^{n-1} - 2 \rho = \rho ( n c \rho^{n-2} - 2 )$, the
only positive critical point of $f_+(\rho)$ is $\rho = \big( \tfrac{2}{nc}
\big)^{\frac{1}{n-2}}$.  From $f_+(0) = \eps > 0$ and $\lim_{\rho \to \infty}
f_+(\rho) = \infty$, we see that $f_+(\rho)$ has two distinct positive
roots if and only if $f_+ \big( \big( \tfrac{2}{nc} \big)^{\frac{1}{n-2}}
\big) < 0$, which is equivalent to $\eps < \eps_\ast$.
Hence, $f_+(\rho)$ has two distinct positive roots $\sqrt{\eps} < \rho_2 < \big(
\tfrac{2}{nc} \big)^{\frac{1}{n-2}} < \rho_3$ if and only if $\eps < \eps_\ast$.
\eop
\end{proof}


\begin{lem} \label{lem:loc_zeros_G}
In the notation of Lemma~\ref{lem:num_zeros_G}, if
\begin{equation}
0 < \eps < \min \big\{ \eps_\ast, \;
( \tfrac{1}{nc} )^{\frac{2}{n-2}} ( \tfrac{n}{n-1} )^{\frac{n}{n-2}}, \;
( \tfrac{1}{c(n-1)} )^{\frac{2}{n-2}} (\tfrac{n-1}{n})^{\frac{n}{n-2}} \big\},
\label{eqn:eps_cond_sharper_bound}
\end{equation}
then $G(w)$ has $3n$ zeros, and we have
\begin{equation*}
\eta^{-1} \sqrt{\eps} < \rho_1 < \sqrt{\eps} < \rho_2 <
\eta \sqrt{\eps}, 
\quad \text{where } \eta \coloneq (\tfrac{n}{n-1})^{\frac{1}{2}}.
\end{equation*}
\end{lem}

\begin{proof}
Recall that $\rho_1$ is the positive root of $f_-(\rho) \coloneq c \rho^n
+ \rho^2 - \eps$.  Note that $f_-(\sqrt{\eps}) > 0$.  We then
have $(\tfrac{n-1}{n})^{\frac{1}{2}} \sqrt{ \eps}
< \rho_1 < \sqrt{\eps}$, if
\begin{equation*}
f_- \Big( (\tfrac{n-1}{n})^{\frac{1}{2}} \sqrt{\eps} \Big)
= c (\tfrac{n-1}{n})^{\frac{n}{2}} \eps^{\frac{n}{2}} + \tfrac{n-1}{n} \eps
- \eps
= c (\tfrac{n-1}{n})^{\frac{n}{2}} \eps^{\frac{n}{2}} -
\tfrac{1}{n} \eps < 0,
\end{equation*}
which holds if and only if $\eps < ( \tfrac{1}{nc} )^{\frac{2}{n-2}} (
\tfrac{n}{n-1} )^{\frac{n}{n-2}}$.
To derive the bound for $\rho_2$, the smaller positive root of
$f_+(\rho) \coloneq c \rho^n - \rho^2 + \eps$, note first that
$f_+(\sqrt{\eps}) = c \sqrt{\eps}^n > 0$.
We then have $\sqrt{\eps} < \rho_2 < (\tfrac{n}{n-1})^{\frac{1}{2}}
\sqrt{\eps}$, if
\begin{equation*}
f_+ \big( (\tfrac{n}{n-1})^{\frac{1}{2}} \sqrt{\eps} \big)
= c (\tfrac{n}{n-1})^{\frac{n}{2}} \eps^{\frac{n}{2}} - \tfrac{n}{n-1} \eps
+ \eps = c
(\tfrac{n}{n-1})^{\frac{n}{2}} \eps^{\frac{n}{2}} - \tfrac{1}{n-1} \eps < 0,
\end{equation*}
which holds if and only if $\eps < ( \tfrac{1}{c(n-1)} )^{\frac{2}{n-2}}
(\tfrac{n-1}{n})^{\frac{n}{n-2}}$.
\eop
\end{proof}

\begin{lem} \label{lem:sense_zeros_G}
Let $\eps > 0$ satisfy condition~\eqref{eqn:eps_cond_sharper_bound}.
Then $G(w)$ in~\eqref{eqn:defn_G} is sense-preserving at its zeros 
$\rho_1 e^{i \frac{(2k+1) \pi}{n}}$ and sense-reversing at its zeros
$\rho_2 e^{i \frac{2k \pi}{n}}$, $1 \leq k \leq n$.
\end{lem}
\begin{proof}
We verify directly Definition~\ref{defn:sense-pres-rev} for both types of zeros.
For ease of notation we denote the analytic part of $G(w)$ by
$R_G(w) \coloneq c w^{n-1} + \tfrac{\eps}{w}$.
We have
\begin{equation*}
R_G'(w) = (n-1) c w^{n-2} - \tfrac{\eps}{w^2} 
= \tfrac{1}{w^2} ( (n-1) c w^n - \eps ).
\end{equation*}
Any zero $w'$ of $G(w)$ is a solution of~\eqref{eqn:simplified_local_eqn},
thus satisfying $c (w')^n = \abs{w'}^2 - \eps$, so that
\begin{equation*}
R_G'(w') = \tfrac{1}{ (w')^2 } ( (n-1) \abs{w'}^2 - (n-1) \eps - \eps )
= \tfrac{1}{ (w')^2 } ( (n-1) \abs{w'}^2 - n \eps )
\end{equation*}
and
\begin{equation*}
\abs{ R_G'(w') } = \abs{ (n-1) - n \tfrac{\eps}{\abs{w'}^2} }.
\end{equation*}
Since $\eps$ satisfies~\eqref{eqn:eps_cond_sharper_bound}, we
have $\rho_1^2 < \eps < \rho_2^2 < \tfrac{n}{n-1} \eps$, see
Lemma~\ref{lem:loc_zeros_G}, which implies
\begin{equation*}
(n-1) - n \tfrac{\eps}{\rho_1^2} < (n-1) - n \tfrac{\eps}{\rho_2^2}
< (n-1) - n \tfrac{n-1}{n} = 0.
\end{equation*}
Then
\begin{equation*}
\abs{ R_G'(\rho_1 e^{i \frac{(2 k + 1) \pi}{n}} ) }
= n \tfrac{\eps}{\rho_1^2} - (n-1) > n - (n-1) = 1,
\end{equation*}
which shows that $G(w)$ is sense-preserving at the zeros
$\rho_1 e^{i \frac{(2k+1) \pi}{n}}$, and
\begin{equation*}
\abs{ R_G'(\rho_2 e^{i \frac{2k \pi}{n}}) }
= n \tfrac{\eps}{\rho_2^2} - (n-1) < n - (n-1) = 1,
\end{equation*}
which shows that $G(w)$ is sense-reversing at the zeros $\rho_2 e^{i \frac{2k
\pi}{n}}$.
\eop
\end{proof}

The preceding lemma concludes the discussion of the zeros of $G(w)$.
In the following main result of this section we prove \textit{(i)} of
Theorem~\ref{thm:pert_new_zeros}.  We will show that the zeros $\rho_1
e^{i \frac{(2k+1) \pi}{n}}$ and $\rho_2 e^{i \frac{2k \pi}{n}}$ of
$G(w)$ give rise to zeros of $F(w)$.

\begin{thm}
\label{thm:F_has_2n_zeros}
Let $n \geq 3$ and $c > 0$ and set $\eta \coloneq
(\tfrac{n}{n-1})^{\frac{1}{2}}$.  Then, for all sufficiently small $\eps > 0$,
the function $F(w)$ in~\eqref{eqn:F_apart} has at least $n$ zeros in
$A(0, \eta^{-1} \sqrt{\eps}, \sqrt{\eps})$, at least $n$ zeros in 
$A(0, \sqrt{\eps}, \eta\sqrt{\eps})$, 
and no zeros in $\overline{D}(0, \eta^{-1} \sqrt{\eps})$.
If $F(w)$ is regular, then at least $n$ of its zeros in
$A(0, \eta^{-1} \sqrt{\eps}, \sqrt{\eps})$ are sense-preserving and at least $n$
of its zeros in $A(0, \sqrt{\eps}, \eta \sqrt{\eps})$ are sense-reversing.
\end{thm}
\begin{proof}
Let $\eps$ satisfy~\eqref{eqn:eps_cond_sharper_bound}, so that in
particular $G(w)$ in~\eqref{eqn:defn_G} has $3n$ zeros.  We will show
that the $2n$ zeros of $G(w)$ discussed in
Lemma~\ref{lem:sense_zeros_G} give rise to $2n$ zeros of $F(w)$.

Consider a sense-preserving zero $w_+ = \rho_1 e^{i \frac{(2k+1) \pi}{n}}$ of
$G(w)$.  From Lemma~\ref{lem:loc_zeros_G} we have $\eta^{-1} \sqrt{\eps}
< \abs{w_+} < \sqrt{\eps}$.  In particular, $w_+$ is the \emph{only}
exceptional point of $G(w)$ in the annular sector defined by the two radii
$\eta^{-1} \sqrt{\eps}$ and $\sqrt{\eps}$, and the two half-lines
$\arg(w) = \tfrac{2k \pi}{n}$ and $\arg(w) = \tfrac{(2k+2) \pi}{n}$;
see Figure~\ref{fig:annular_sector}.  Let $\Gamma^+ = [\Gamma^+_1,
\Gamma^+_2, \Gamma^+_3, \Gamma^+_4]$ be the boundary curve of this
annular sector as indicated in Figure~\ref{fig:annular_sector}.  Since $G(w)$ is
sense-preserving at $w_+$ (see Lemma~\ref{lem:sense_zeros_G}), we know that
$V(G; \Gamma^+) = +1$ and we will show next that $V(F; \Gamma^+) = +1$.

In order to apply Rouch\'e's theorem
(Theorem~\ref{thm:WindingRouche}), we will show that
\begin{equation}
\label{eq:rouche_goal}
    \abs{F(w) - G(w)} < \abs{F(w)} + \abs{G(w)}, \quad w \in \Gamma^+.
\end{equation}
From~\eqref{eqn:F_apart} and~\eqref{eqn:defn_G}, we see that inside a disk
around $w = 0$ contained in the domain of convergence of the series of $R(w)$
we have $\abs{F(w) - G(w)} \leq M\abs{w}^n$, for some $M>0$ independent of
$\eps$.  Hence, for sufficiently small $\eps>0$, $\Gamma^+$ is inside
this disk.  Thus, it suffices to show that $\abs{G(w)} > M \abs{w}^n$
on the arcs that compose $\Gamma^+$ (trivially, $\abs{F(w)} +
\abs{G(w)} \ge \abs{G(w)}$).

For $w \in \Gamma^+_1$ we have $\abs{w} = \sqrt{\eps}$, so that for sufficiently small 
$\eps>0$,
\begin{equation}
\label{eq:sqrteps_circ}
    \abs{G(w)}
    = \abs{w}^{-1} \abs{cw^n + \eps - \abs{w}^2}
    = c \sqrt{\eps}^{n-1}
    > M \sqrt{\eps}^n
    = M \abs{w}^n.
\end{equation}
For $w \in \Gamma^+_2$ or $w \in \Gamma^+_4$ we have $w^n = \abs{w}^n$
and $\eta^{-1} \sqrt{\eps} \leq \abs{w} \leq \sqrt{\eps}$.
If $\eps>0$ is sufficiently small, then
\begin{equation*}
\begin{split}
    \abs{G(w)}
    & = \abs{w}^{-1} \abs{cw^n + \eps - \abs{w}^2}
    = \abs{w}^{-1} (c \abs{w}^n + \eps -  \abs{w}^2)
    \ge c \abs{w}^{n-1}\\
    & \ge c \eta^{-(n-1)} \sqrt{\eps}^{n-1}
    > M \sqrt{\eps}^n
    \ge M \abs{w}^n.
\end{split}
\end{equation*}
For $w \in \Gamma^+_3$ we have $\abs{w} = \eta^{-1} \sqrt{\eps}$, so that for
sufficiently small $\eps>0$,
\begin{equation}
\label{eq:inner_circ}
\begin{split}
    \abs{G(w)}
    & = \abs{w}^{-1} \abs{cw^n + \eps - \abs{w}^2}
    = \abs{w}^{-1} \abs{c w^n + (1 - \eta^{-2})\eps}\\
    & \ge \abs{w}^{-1} ( (1 - \eta^{-2}) \eps - c \abs{w}^n )
    \ge \tfrac{1 - \eta^{-2}}{\eta^{-1}} \sqrt{\eps} - c \eta^{-(n-1)}
        \sqrt{\eps}^{n-1}\\
    & > M \eta^{-n} \sqrt{\eps}^n = M \abs{w}^n.
\end{split}
\end{equation}
Hence for sufficiently small $\eps>0$ we find
that~\eqref{eq:rouche_goal} is satisfied on $\Gamma^+$ and thus $V(F;
\Gamma^+) = 1$, so $F(w)$ has at least one zero inside this sector (by
the argument principle).  If $F(w)$ is regular, this zero is
sense-preserving.  Since $G(w)$ has $n$ zeros of type $w_+$,
$F(w)$ has at least $n$ such zeros in the annulus
$\annu(0, \eta^{-1} \sqrt{\eps}, \sqrt{\eps})$.

We can use exactly the same reasoning as above to show that the
sense-reversing zeros $\rho_2 e^{i \frac{2k \pi}{n}}$ of $G(w)$ give
zeros of $F(w)$.  From Lemma~\ref{lem:num_zeros_G} and
Lemma~\ref{lem:loc_zeros_G} we see that $G(w)$ has $n$ such zeros
inside $\annu(0, \sqrt{\eps}, \eta \sqrt{\eps})$.  Now,
fix $w_- = \rho_2 e^{i \frac{2k \pi}{n}}$.  Consider the boundary
curve $\Gamma^- = [\Gamma^-_1,\Gamma^-_2, \Gamma^-_3, \Gamma^-_4]$ of
the annular sector defined by the radii $\sqrt{\eps}$ and $\eta
\sqrt{\eps}$ and the half-lines $\arg(w) = \tfrac{(2k-1) \pi}{n}$ and
$\arg(w) = \tfrac{(2k+1) \pi}{n}$ (again, see
Figure~\ref{fig:annular_sector}).  As before, we show that $\abs{G(w)} > M
\abs{w}^n$ on $\Gamma^-$.

The arc $\Gamma^-_3$ has been treated already
in~\eqref{eq:sqrteps_circ}.  For $w \in \Gamma^-_1$ we have $\abs{w} =
\eta \sqrt{\eps}$, so that for sufficiently small $\eps>0$,
\begin{equation}
\label{eq:outer_circ}
\begin{split}
    \abs{G(w)}
    & = \abs{w}^{-1} \abs{cw^n + \eps - \abs{w}^2}
    = \abs{w}^{-1} \abs{c w^n + (1 - \eta^2)\eps}\\
    & \ge \abs{w}^{-1} ( (\eta^2 - 1)\eps - c \abs{w}^n )
    \ge \tfrac{\eta^2 - 1}{\eta} \sqrt{\eps} - c \eta^{n-1}
        \sqrt{\eps}^{n-1}\\
        & > M \eta^n \sqrt{\eps}^n = M \abs{w}^n.
\end{split}
\end{equation}
For $w \in \Gamma^-_2$ or $w \in \Gamma_4^-$ we have $w^n = - \abs{w}^n$.
Using that $\sqrt{\eps} \le \abs{w} \le \eta \sqrt{\eps}$ we compute,
for a sufficiently small $\eps>0$,
\begin{equation*}
\begin{split}
    \abs{G(w)}
    & = \abs{w}^{-1} \abs{cw^n + \eps - \abs{w}^2}
    = \abs{w}^{-1} \abs{- c \abs{w}^n + \eps - \abs{w}^2}\\
    & = \abs{w}^{-1} (c \abs{w}^n + \abs{w}^2 - \eps)
    \ge c \abs{w}^{n-1}
    \ge c \sqrt{\eps}^{n-1}
    > M \eta^n \sqrt{\eps}^n
    \ge M \abs{w}^n.
\end{split}
\end{equation*}
As before we can now conclude that $V(F; \Gamma^-) = -1$, so $F(w)$
has at least one zero inside $\Gamma^-$, which is sense-reversing if
$F(w)$ is regular.  In total, $n$ such zeros exist. 

In order to complete the proof, let $\gamma$ denote the
circle $\abs{w} =\eta^{-1} \sqrt{\eps}$. From
the computation~\eqref{eq:inner_circ} we see that Rouch\'e's theorem
applies to $F(w)$ and $G(w)$ on $\gamma$.  Since $G(w)$ has a simple
pole and no zeros inside $\gamma$, we obtain $V(F; \gamma) =
V(G; \gamma) = -1$.

Let $R_F(w) \coloneq R(w) + \tfrac{\eps}{w}$.
For $0 \neq w \in \cdisk(0, \eta^{-1} \sqrt{\eps})$ we compute
\begin{equation*}
\abs{R_F'(w)}
= \abs{R'(w) - \tfrac{\eps}{w^2}}
\ge \tfrac{\eps}{\abs{w}^2} - \abs{R'(w)}
\ge \eta^2 - \abs{R'(w)} > 1
\end{equation*}
for sufficiently small $\eps>0$, since $R'(0) = 0$.  This shows that
$F(w)$ is sense-preserving on the disk $\cdisk(0,
\eta^{-1} \sqrt{\eps})$.  Since $0$ is a simple pole of $F(w)$ with Poincar\'{e}
index $-1 = V(F; \gamma)$, there are no zeros of $F(w)$ in this disk.
\eop
\end{proof}

\begin{figure}[t]
\subfigure[]{%
\label{fig:annular_sector}%
\includegraphics[width=.54\textwidth]{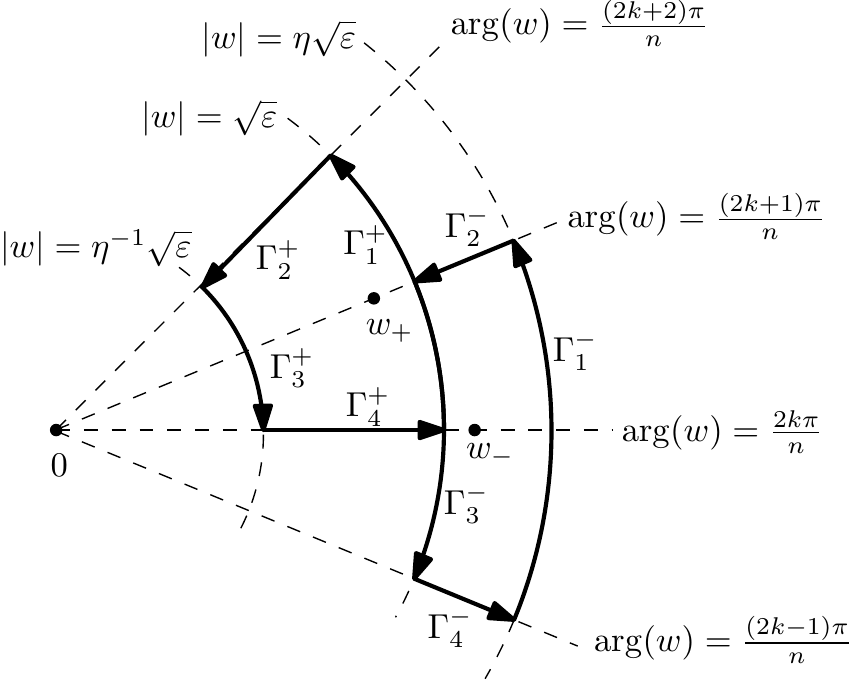}
}
\hfill
\subfigure[$\gamma$ denotes the circle $\abs{z-z_0} = \eta \sqrt{\eps}$]{%
\label{fig:discs}%
\includegraphics[width=.44\textwidth]{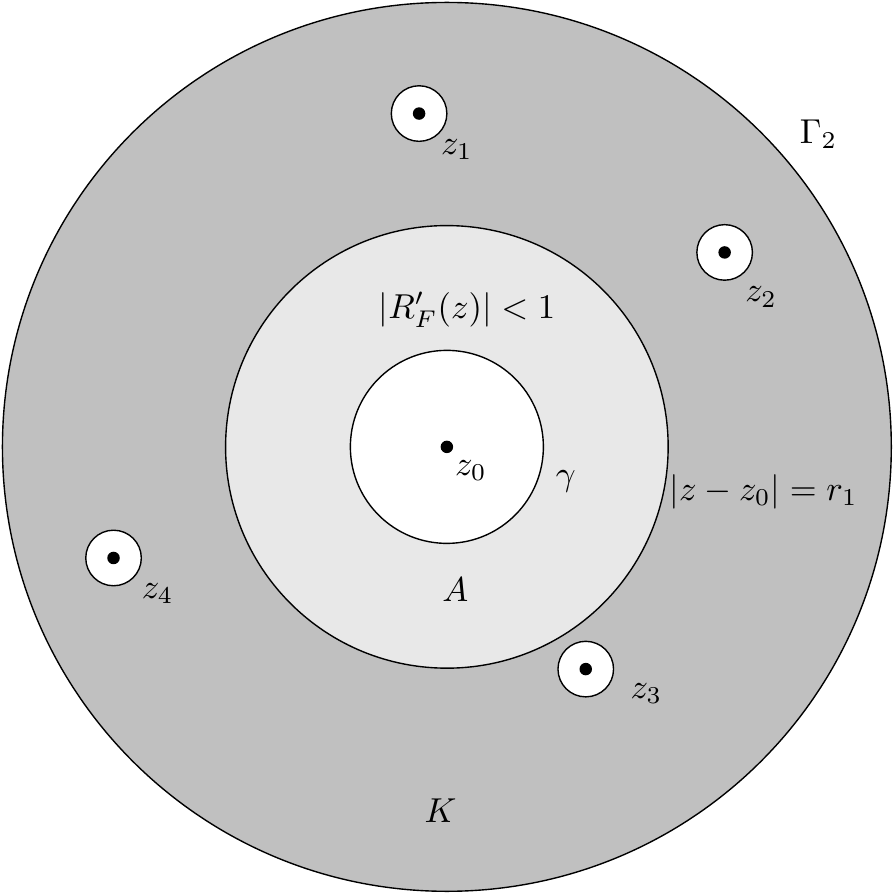}
}
\caption{Illustrations for the proofs of Theorems~\ref{thm:F_has_2n_zeros}
(left) and~\ref{thm:F_has_no_other_zeros} (right).}
\end{figure}

\begin{rem}
\label{rem:not_sharp}
Theorem~\ref{thm:F_has_2n_zeros} gives a lower bound for $2n$ additional
zeros close to $0$.  However, the proof does not show that there are
\emph{exactly} $2n$ such zeros, because the sector enclosed by
$\Gamma^+$ may contain an open region where $\abs{R_F'(w)} < 1$.  In that
case there may exist additional pairs of zeros inside this sector so
that the total winding of $+1$ is maintained (similarly for
$\Gamma^-$).
We suspect that for sufficiently small $\eps>0$ exactly $2n$ additional zeros
occur; but see also Section~\ref{sec:example_smallder}, where more than $2n$
additional zeros occur for somewhat larger $\eps$.
\end{rem}
\medskip

We end this section with a result that follows easily from the proof of
Theorem~\ref{thm:F_has_2n_zeros} and that will be helpful in the proof of 
Theorem~\ref{thm:F_has_no_other_zeros} below.

\begin{cor}\label{cor:windF}
Under the assumptions of Theorem~\ref{thm:F_has_2n_zeros}, we have
$V(F; \gamma) = -1$, where $\gamma$ is the circle $\abs{w} = \eta \sqrt{\eps}$.
\end{cor}

\begin{proof}
From~\eqref{eq:outer_circ} we see that Rouch\'e's theorem
(Theorem~\ref{thm:WindingRouche}) applies to
$F(w)$ and $G(w)$ on the circle $\gamma$.  But $G(w)$ has exactly
one simple pole, $n$ sense-preserving and $n$ sense-reversing zeros
inside $\gamma$, so $V(F; \gamma) = V(G; \gamma) = -1$.
\eop 
\end{proof}

\subsection{Behavior away from $z_0$ (Proof of \textit{(ii)} and \textit{(iii)}
in Theorem~\ref{thm:pert_new_zeros})}
\label{sec:away_z0}

In the previous section we substituted $w = z-z_0$ because it simplified 
the notation. This is no longer necessary, so we work in the $z$ variable again.
We will first prove \textit{(ii)} of Theorem~\ref{thm:pert_new_zeros}, but for
a slightly more general setting, where we do not require $f(z_0) = 0$.  This
setting will be needed in the proof of Theorem~\ref{thm:other_z0}.

\begin{thm}
\label{thm:pert_local}
Let $f(z) = R(z) - \conj{z}$ with $\deg(R) \geq 2$, and $z_0 \in \C$.  Let
$z_1, \ldots, z_N \in \C$ be the regular zeros of
$f(z)$, except, possibly, $z_0$.
Then there exist mutually disjoint disks $\cdisk(z_k,r)$ not containing $z_0$
with the following property:
For all sufficiently small $\eps>0$ the function
\begin{equation*}
F(z) = f(z) + \tfrac{\eps}{z-z_0}
\end{equation*}
has exactly one zero in $\disk(z_k,r)$, and the index of this zero is
$\ind(z_k; f)$.
\end{thm}
\begin{proof}
Let $z_1, \ldots, z_M$, $M \geq N$, be all zeros of $f(z)$, except, possibly,
$z_0$. Choose $r > 0$ with the following properties:
\begin{compactenum}
\item The disks $\cdisk(z_k,r)$ do not intersect ($k = 1, \ldots, M$)
and do neither contain $z_0$ nor the poles of $R(z)$.
\item If $z_k$ is a regular zero, $f(z)$ is either sense-preserving or
sense-reversing on $\cdisk(z_k,r)$.
\end{compactenum}
Fix a regular zero $z_k$, $1 \leq k \leq N$.  Let $\Gamma$ denote the circle
around $z_k$
with radius $r$.  By construction, the continuous function $z \mapsto \abs{f(z)
(z-z_0)}$ admits a positive minimum on the compact set $\Gamma$.  For any
\begin{equation*}
0 < \eps < \tfrac{1}{2} \min_{z \in \Gamma} \abs{ f(z) (z-z_0) }
\end{equation*}
we then have
\begin{equation*}
\abs{ F(z) } = \abs{ f(z) + \tfrac{\eps}{z-z_0} } \geq \abs{f(z)} - \abs{
\tfrac{\eps}{z-z_0} } > \tfrac{\eps}{ \abs{z-z_0} }, \quad z \in \Gamma,
\end{equation*}
from which we find $\abs{F(z)-f(z)} < \abs{F(z)} + \abs{f(z)}$ for $z \in
\Gamma$.  By Rouch\'e's theorem (Theorem~\ref{thm:WindingRouche}) we have
$V(F; \Gamma) = V(-f; \Gamma) = \pm 1$, since $f(z)$ has
exactly one regular zero in the interior of $\Gamma$. Thus $F(z)$ also
has (at least) one zero in the interior of $\Gamma$ (by 
Theorem~\ref{thm:arg_principle}; see also the degree 
principle~\cite[Section~2.3.6]{Sheil-Small2002}
or~\cite[Theorem~2.3]{Balk1991}).

Write $R_F(z) = R(z) + \tfrac{\eps}{z-z_0}$, so that $F(z) = R_F(z) - \conj{z}$.
Now suppose that $f(z)$ is sense-preserving at $z_k$, so that
$\abs{R'(z)} > 1$ on $\cdisk(z_k,r)$.  Note that $R_F'(z) =
R'(z) - \tfrac{\eps}{(z-z_0)^2}$, so that $\abs{R_F'(z)} > 1$ on the disk
if $\eps$ is chosen sufficiently small. Therefore $F(z)$ is also
sense-preserving and its zeros have positive index ($+1$, see
Proposition~\ref{prop:P_index_for_zeros_poles}).  Then $V(F; \Gamma) = V(f;
\Gamma) = \ind(z_k; f) = +1$ shows that $F(z)$ has exactly one zero in this
disk.  A similar reasoning holds for sense-reversing zeros of $f(z)$.
\eop
\end{proof}

\begin{rem}
If all zeros of $f(z)$ have merely nonzero index, the first part of the proof
still shows that $F(z)$ has \emph{at least} one zero near every zero $z_k \neq z_0$ of
$f(z)$.
\end{rem}

\begin{lem} \label{lem:no_zero}
Let $f(z) = R(z) - \conj{z}$ with $\deg(R)\geq 2$.  Let $z_0, \zeta \in \C$ with
$z_0 \neq \zeta$ and $f(\zeta) \neq 0$.  Then there exists a neighborhood of
$\zeta$ in which, for all sufficiently small $\eps > 0$, the function
\begin{equation*}
F(z) = f(z) + \tfrac{\eps}{z-z_0}
\end{equation*}
has no zeros.
\end{lem}
\begin{proof}
Let us assume that $\zeta$ is not a pole of $R(z)$, else there is nothing to
show.  Then there exists $0 < r < \abs{z_0-\zeta}$ such that $f(z)$ is
continuous and nonzero on $\cdisk(\zeta, r)$.  For
\begin{equation*}
\eps < \min_{ z \in \cdisk(\zeta,r) } \abs{ (z-z_0) f(z)},
\end{equation*}
(the right hand side is positive by the choice of $r$), we see from
\begin{equation*}
\abs{F(z)} \geq \abs{f(z)} - \tfrac{\eps}{\abs{z-z_0}}
\end{equation*}
that $F(z)$ has no zeros in $\cdisk(\zeta,r)$.
\eop
\end{proof}

The following theorem completes the discussion of points ``away'' from
the perturbation point $z_0$.  Together with Theorem~\ref{thm:F_has_2n_zeros}
and Theorem~\ref{thm:pert_local}, it implies
Theorem~\ref{thm:pert_new_zeros}.

\begin{thm}
\label{thm:F_has_no_other_zeros}
In the notation and under the assumptions of Theorem~\ref{thm:pert_new_zeros},
let additionally $f(z)$ be regular. Then for sufficiently small $\eps > 0$, the
functions $F(z)$ and $f(z)$ have the same number of zeros outside $\cdisk(z_0,
\eta \sqrt{\eps})$.  More precisely:
\begin{compactenum}
\item If $z_1, \ldots, z_N$ are the zeros of $f(z)$ with $\abs{z_k - z_0} > \eta
\sqrt{\eps}$, then there exist mutually disjoint disks $\cdisk(z_k,r)$, $1 \leq
k \leq N$, such that $F(z)$ has exactly one zero in each $\disk(z_k,r)$,
and the index of this zero is $\ind(z_k; f)$.
\item $F(z)$ has no further zeros outside $\cdisk(z_0, \eta \sqrt{\eps})$.
\end{compactenum}
\end{thm}
\begin{proof}
Since $\lim_{z \to \infty} f(z) = \infty$, there exists $r_2 > 0$ such that
$\abs{f(z)} \geq 1$ for $\abs{z-z_0} \geq r_2$.  Further we can choose $r_2$ so
that $\Gamma_2 = \{ z :  \abs{z-z_0} = r_2 \}$ contains all poles of $f(z)$ in
its interior.  For $\abs{z-z_0} \geq r_2$ we have
\begin{equation*}
\abs{F(z)} = \abs{ f(z) + \tfrac{\eps}{z-z_0} }
\geq \abs{f(z)} - \tfrac{\eps}{\abs{z-z_0}}
\geq 1 - \tfrac{\eps}{r_2},
\end{equation*}
which is positive for sufficiently small $\eps>0$.  Thus for each such $\eps$
the function $F(z)$ has no zeros on or exterior to $\Gamma_2$.  We further have
\begin{equation*}
\abs{F(z) - f(z)} = \tfrac{\eps}{\abs{z-z_0}} = \tfrac{\eps}{r_2}
< 1 \leq \abs{f(z)} \leq \abs{F(z)} + \abs{f(z)}, \quad z \in \Gamma_2.
\end{equation*}
Hence Rouch\'e's theorem (Theorem~\ref{thm:WindingRouche}) implies
$V(F; \Gamma_2) = V(f; \Gamma_2)$. By Theorem~\ref{thm:pert_local} each
zero of $f(z)$ inside $\Gamma_2$ has a corresponding zero of $F(z)$ with same
index (except for $z_0$).  We will now show that if $\eps > 0$ is sufficiently 
small, then $F(z)$
has no further zeros inside $\Gamma_2$ than implied by
Theorem~\ref{thm:F_has_2n_zeros} and Theorem~\ref{thm:pert_local}. In order
to achieve this, we will use Lemma~\ref{lem:no_zero} combined with a compactness
argument.

By assumption $R'(z_0)=0$.  Thus we can choose
$r_1 > 0$ such that $\abs{R'(z)} < \frac{1}{n}$ for $\abs{z-z_0} \leq r_1$, and
such that $z_0$ is the only zero of $f(z)$ in $\abs{z-z_0} \leq r_1$.
We define the compact set $K$ as follows.  Consider the closed annulus
$r_1 \le \abs{z-z_0} \le r_2$ and denote by $z_1, \dotsc, z_N$ the zeros
of $f(z)$ in that annulus.  By Theorem~\ref{thm:pert_local}, there are
mutually disjoint $\cdisk(z_k,r)$, $0 \leq k \leq N$, such that $F(z)$ has
exactly one zero $z_k' \in \disk(z_k,r)$ of the same index, $1 \leq k \leq N$.
Cutting out these neighborhoods in the annulus, we obtain the compact set $K$
(Figure~\ref{fig:discs}).

For each $\zeta \in K$ we have $f(\zeta) \neq 0$, so there exists a
neighborhood of $\zeta$ as in Lemma~\ref{lem:no_zero}.  These neighborhoods
constitute an open covering of $K$, of which a finite subset is sufficient to
cover $K$.  On each neighborhood $F(z)$ is nonzero for all sufficiently small
$\eps$; see Lemma~\ref{lem:no_zero}.  Hence, only a finite number of smallness
constraints on $\eps$ are sufficient to guarantee that $F(z)$ has no zeros
inside $K$.

Recall that inside each of the cut out disks $\disk(z_k,r)$, $1 \leq k \leq N$,
the function $F(z)$ has exactly one zero of same index as $f(z)$.  Thus, it
remains to show that $F(z)$ has no additional zeros inside the annulus
$A \coloneq A(z_0, \eta \sqrt{\eps}, r_1)$.  As before set $R_F(z) = R(z) +
\tfrac{\eps}{z-z_0}$.  Then, for $z \in A$,
\begin{equation*}
\abs{R_F'(z)} \le \abs{R'(z)} + \tfrac{\eps}{\abs{z-z_0}^2}
< \tfrac{1}{n} + \tfrac{n-1}{n} = 1,
\end{equation*}
and thus $F(z)$ is sense-reversing on $A$.  Denote by $\gamma$ the circle
$\{ z : \abs{z-z_0} = \eta \sqrt{\eps} \}$, and by $n_-$ the number of
(sense-reversing) zeros of $F(z)$ in $A$.  By the argument principle we find
\begin{equation*}
-1 + \sum_{k=1}^N \ind(z_k; f) = V(f; \Gamma_2) = V(F; \Gamma_2)
= V(F; \gamma) + \sum_{k=1}^N \ind(z_k'; F) - n_-,
\end{equation*}
which shows $n_- = 0$, since $\ind(z_k; f) = \ind(z_k'; F)$ for $k =
1, \ldots, N$, and $V(F; \gamma) = -1$; see Corollary~\ref{cor:windF}.
\eop
\end{proof}

\subsection{Perturbation at arbitrary points}
\label{sec:other_z0}

We now consider perturbations at points where the assumptions of 
Theorem~\ref{thm:pert_new_zeros} are not satisfied. The situation is
simpler than in the setting of Theorem~\ref{thm:pert_new_zeros} and
for the proof the same techniques as in Section~\ref{sec:near_z0} can 
be applied. We therefore only give a sketch of the proof. Furthermore,
we assume for simplicity that both $f(z)$ and $F(z)$ are regular, 
although this requirement could be weakened somewhat.

\begin{thm}
\label{thm:other_z0}
Let $f(z) = R(z) - \conj{z}$ with $\deg(R)\geq 2$ be regular and satisfy
$\lim_{z \to \infty} f(z) = \infty$, and let $z_0 \in \C$.
For sufficiently small $\eps > 0$, if
\begin{equation*}
F(z)  = f(z) + \tfrac{\eps}{z-z_0}
\end{equation*}
is regular, then the following holds:
\begin{compactenum}
\item
\label{it:at_poles}
If $z_0$ is a pole of $f(z)$, then $F(z)$ and $f(z)$ have the same number of
zeros.

\item
\label{it:at_nonzero}
If $0 < \abs{f(z_0)} < \infty$, then there exists $r > 0$ such that $0 <
\abs{f(z)} < \infty$ on $\disk(z_0,r)$, and $F(z)$ has at least one
sense-preserving zero in $\disk(z_0,r)$.

\item
\label{it:at_preserve}
If $f(z_0) = 0$, and $\abs{R'(z_0)} > 1$, there exists $r > 0$ such that $F(z)$
has at least two sense-preserving zeros in $\disk(z_0,r)$.

\item
\label{it:at_reverse}
If $f(z_0) = 0$, and $0 < \abs{R'(z_0)} < 1$, there exists $r > 0$ such that
$F(z)$ has at least two sense-preserving zeros and two sense-reversing zeros
in $\disk(z_0,r)$.
\end{compactenum}
Further, in \textit{\ref{it:at_nonzero}.}, \textit{\ref{it:at_preserve}.} and
\textit{\ref{it:at_reverse}.}, $F(z)$ and $f(z)$ have the same number of zeros
outside $\disk(z_0,r)$.
\end{thm}

\begin{proof} (Sketch)

In~\textit{\ref{it:at_poles}.} we have $R(z) =(z-z_0)^{-m} \widetilde{R}(z)$ 
with $\widetilde{R}(z_0) \notin \{ 0, \infty \}$, $m \geq 1$.  In some disk
$\cdisk(z_0,\delta)$ we have $\abs{z} \leq M$ and $\abs{\widetilde{R}(z)} \geq
M' > 0$.  If need be, reduce $\delta$ such that $\delta^m < \tfrac{M'}{2M}$
holds.  Noting that
\begin{equation*}
R_F(z) = R(z) + \tfrac{\eps}{z-z_0} 
= \tfrac{1}{(z-z_0)^m} ( \widetilde{R}(z) + \eps (z-z_0)^{m-1} ),
\end{equation*}
one can compute that $\abs{F(z)} > 0$ on $\cdisk(z_0, \delta)$ for all
sufficiently small $\eps$.  Consider the annulus $\annu(z_0, \delta, r_2)$ which
contains all zeros of $f(z)$ ($r_2$ is chosen as in the proof of
Theorem~\ref{thm:F_has_no_other_zeros}). After cutting out appropriate disks
about the zeros of $f(z)$, a compact set $K$ remains.  As in the proof of
Theorem~\ref{thm:F_has_no_other_zeros}, for sufficiently small $\eps > 0$, each
disk around a zero of $f(z)$ contains exactly one zero of $F(z)$, and the set
$K$ does not contain any zeros of $F(z)$.

In~\textit{\ref{it:at_nonzero}.}
there exists $r > 0$ such that $0 < \abs{f(z)} < \infty$ on $\cdisk(z_0,r)$.  By
the argument principle the winding of $f(z)$ along the boundary curve is $0$. 
Apply Rouch\'e's theorem (Theorem~\ref{thm:WindingRouche}) to $f(z)$ and $F(z)$
on the boundary to see that $F(z)$ has at least one sense-preserving zero inside
this disk, provided $\eps > 0$ is sufficiently small.

Let $r_2$ and $\Gamma_2$ be defined as in the proof of
Theorem~\ref{thm:F_has_no_other_zeros}, and construct the compact set $K$ by
removing from the annulus $\cannu(z_0, r, r_2)$ suitable disks centered at the
zeros of $f(z)$.  Proceeding exactly as in the proof of
Theorem~\ref{thm:F_has_no_other_zeros} we see that $F(z)$ and $f(z)$ have the
same number of zeros outside $\disk(z_0,r)$.

In both~\textit{\ref{it:at_preserve}.} and~\textit{\ref{it:at_reverse}.} we
substitute $w \coloneq z-z_0$ for simplicity of notation, as in
Section~\ref{sec:near_z0}.  Set $c \coloneq R'(0)$, which can be assumed to be
real positive.  Truncation of the Laurent series of the analytic part of
$F(w) = f(w) + \tfrac{\eps}{w}$ yields the function $G(w) = cw + \frac{\eps}{w}
- \conj{w}$.

In~\textit{\ref{it:at_preserve}.} we have $c > 1$ and one computes that
$G(w)$ has exactly two sense-preserving zeros $\pm i(\frac{\eps}{1 +
c})^{\frac{1}{2}}$.  Applying Rouch\'e's theorem to $F(w)$ and $G(w)$ on the
circle $\abs{w} = \sqrt{\eps}$ gives the result.  
An example is shown in Section~\ref{sec:example_preserve}.

Since $c > 1$, there exists $r_1 > 0$ such that $\abs{R'(w)} > \tfrac{1+c}{2}$
on $\cdisk(0,r_1)$.  We then have for $w \in A \coloneq A(0,\sqrt{\eps},r_1)$
\begin{equation*}
\abs{R_F'(w)} = \abs{R'(w) - \tfrac{\eps}{w^2}}
\geq \abs{R'(w)} - \tfrac{\eps}{\abs{w}^2} \geq 1 + \tfrac{c-1}{2} -
\tfrac{\eps}{r_1^2},
\end{equation*}
which is larger than $1$ if $\eps$ is chosen sufficiently small.  Thus $F(w)$
is sense-preserving on $A$.  Now, the same reasoning as in the proof of
Theorem~\ref{thm:F_has_no_other_zeros} (with the modified $A$) shows that $F(w)$
and $f(w)$ have the same number of zeros outside $\disk(z_0,r)$, where $r =
\sqrt{\eps}$.

In~\textit{\ref{it:at_reverse}.} we have $0 < c < 1$ and one computes
that $G(w)$ has exactly the two sense-preserving zeros $\pm
i (\frac{\eps}{1+c})^{\frac{1}{2}}$, and the two sense-reversing zeros $\pm
(\frac{\eps}{1-c})^{\frac{1}{2}}$.  Applying Rouch\'e's theorem to $F(w)$ and
$G(w)$ on $\gamma_1 = \{ w \in \C : \abs{w} = \sqrt{\eps}\}$ shows that $F(w)$
has two sense-preserving zeros inside $\gamma_1$, and applying it on the
circle $\gamma_2 = \{ w \in \C : \abs{w} = \frac{2}{\sqrt{1-c}} \sqrt{\eps} \}$
shows that $F(w)$ has two sense-reversing zeros between $\gamma_1$ and
$\gamma_2$.  An example for this case can be seen in
Section~\ref{sec:example_randpert}, more specifically in
Figure~\ref{fig:randpert2}.

Since $0 < c < 1$ there exists $r_1 > 0$ such that $\abs{R'(w)} <
\tfrac{1+c}{2}$ on $\cdisk(0,r_1)$.  Then, we have for $w \in A \coloneq A(0,
\frac{2}{\sqrt{1-c}} \sqrt{\eps}, r_1)$
\begin{equation*}
\abs{R_F'(w)} = \abs{R'(w) - \tfrac{\eps}{w^2}} \leq \abs{R'(w)} +
\tfrac{\eps}{\abs{w}^2} \leq \tfrac{1+c}{2} + \eps \tfrac{1-c}{4 \eps} < 1,
\end{equation*}
so that $F(w)$ is sense-reversing in $A$.  The same reasoning as in the
proof of Theorem~\ref{thm:F_has_no_other_zeros} now shows that $F(w)$ and $f(w)$
have the same number of zeros outside $\disk(z_0,r)$, where $r =
\frac{2}{\sqrt{1-c}} \sqrt{\eps}$.
\eop
\end{proof}

\section{Examples}
\label{sec:examples}

We briefly discuss the Poincar\'{e} index and its connection with
phase portraits.  (Recall the example given in the Introduction; see
Figure~\ref{fig:circlens}.) Let $z_0$ be an isolated exceptional point
of the continuous function $f(z)$, and let $\gamma$ be a circle around
$z_0$ suitable for the computation of $\ind(z_0; f)$.  Clearly, the
Poincar\'{e} index measures the overall change of the argument of
$f(z)$ while $z$ travels once around $\gamma$ (in the positive sense).
This corresponds exactly to the \emph{chromatic number} of $\gamma$, as
discussed in~\cite[p.~772]{WegertSemmler2011}.  Thus, less formally,
the Poincar\'{e} index corresponds to the number of times each color
appears in the phase portrait while travelling once around $z_0$, and the
sign of the Poincar\'{e} index is revealed by the ordering in which the
colors appear.  This observation allows to determine the Poincar\'{e}
index of an isolated exceptional point of $f(z)$ by looking at the
phase portrait of $f(z)$.

For the color scheme we use in the phase portraits, the color ordering
while travelling around some point is exemplified for the
indices $+1$, $-1$ and $-2$ as follows (left to right):
\begin{center}
\includegraphics[width=36pt]{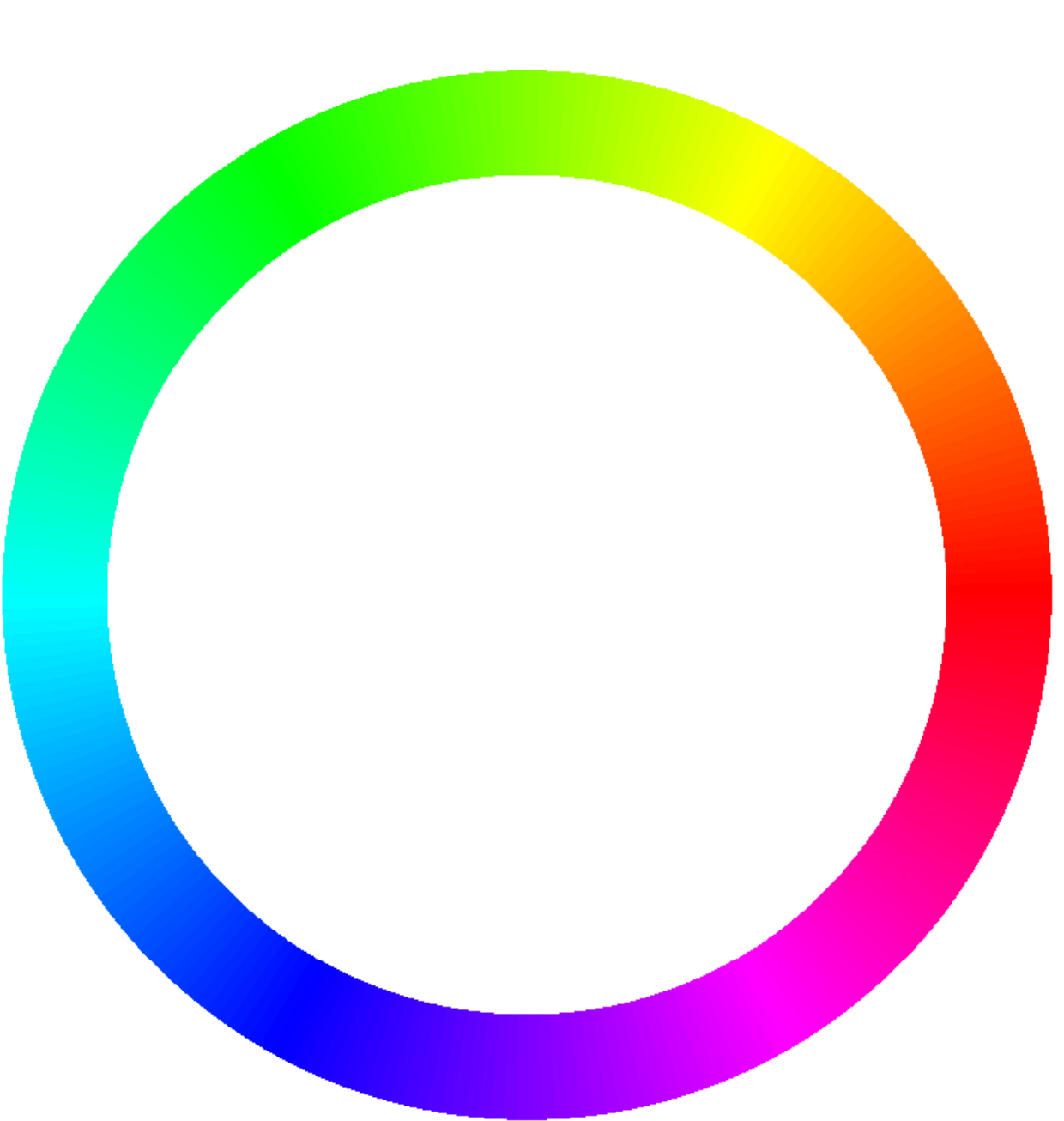}
\qquad
\includegraphics[width=36pt]{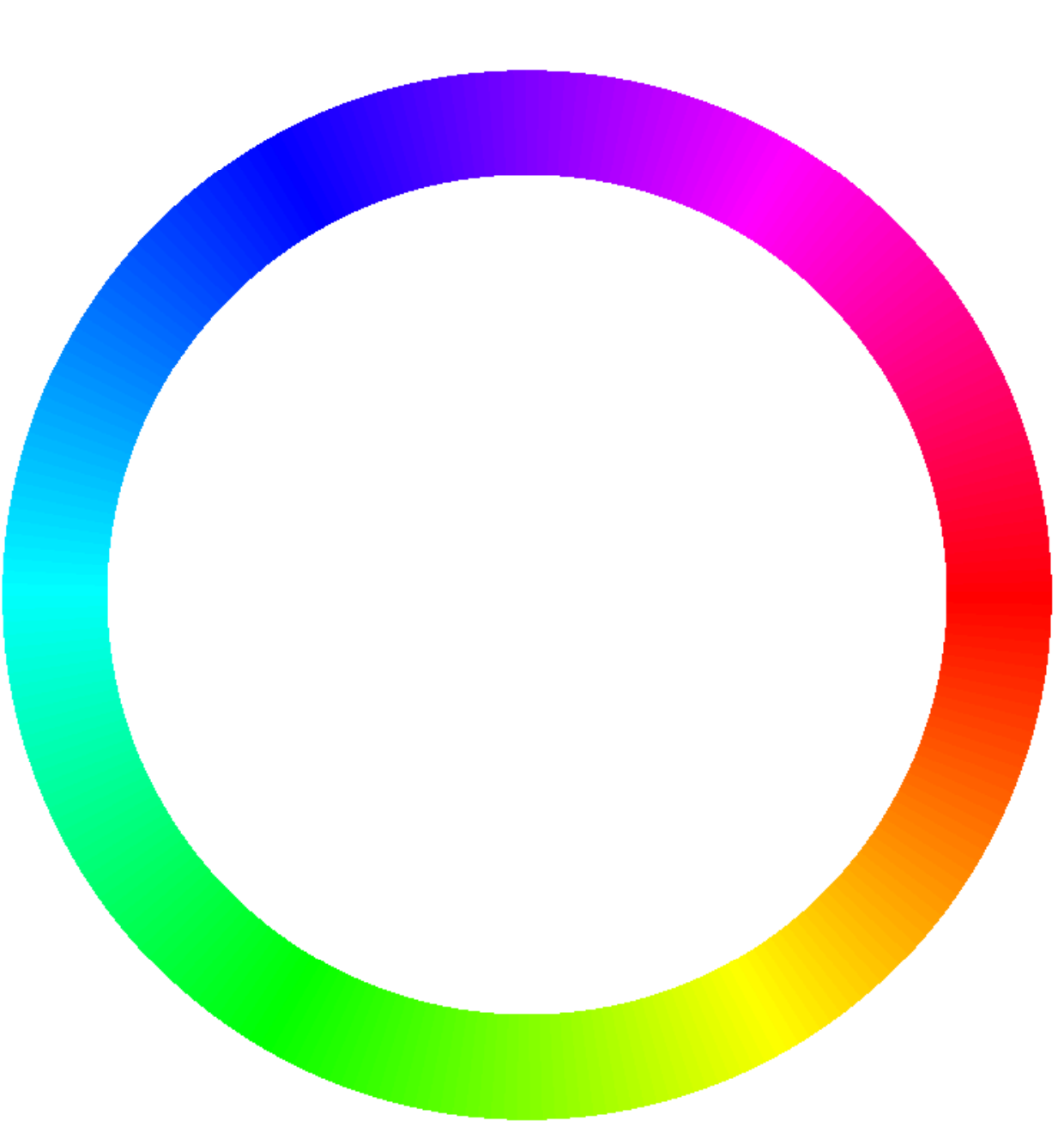}
\qquad
\includegraphics[width=36pt]{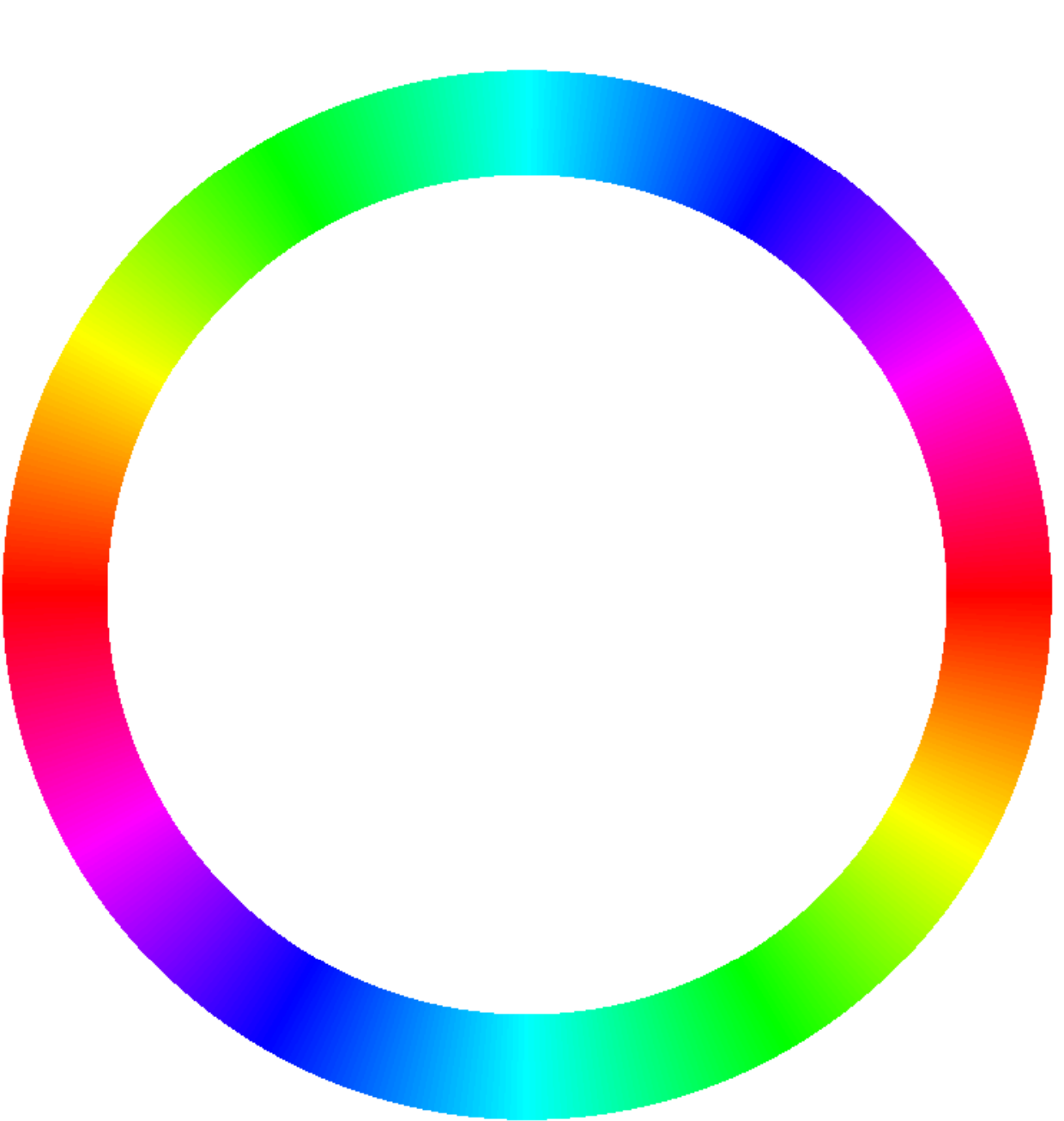}
\end{center}
This is the same ``color wheel'' as used, for example,
in~\cite{WegertSemmler2011}.  For the computation of the phase
portraits in this paper we used a slightly modified version of the
MATLAB package ``Phase Plots of Complex Functions'' by Elias
Wegert\footnote{\url{http://www.visual.wegert.com/}}.

Thus the index of an exceptional point of $f(z) = R(z) - \conj{z}$ can
be determined from the phase portrait.  Poles and zeros of $f(z)$,
which are the only exceptional points, are marked with white squares
(poles) or black disks (zeros).  Additionally, all phase portraits are
visually divided in slightly brightened regions, where $f(z)$ is
sense-preserving (i.e., the analytic part dominates), and slightly
darkened regions, where $f(z)$ is sense-reversing (i.e., the
co-analytic part dominates).  Notice that the indices of zeros in
the latter case are always $-1$, while the indices of sense-preserving
zeros are $+1$.

\subsection{Circular point lenses}
\label{sec:example_circlens}

Recall Rhie's construction described in the Introduction, which
is based on the function $R_0(z) = \frac{z^{d-1}}{z^d - r^d}$, for some 
$d \geq 2$ and $r > 0$. It is easy to see that $R_0^{(k)}(0) = 0$ for $1 \leq k
\leq d-2$, while $R_0^{(d-1)}(0) \neq 0$.  An application of \textit{(i)} in
Theorem~\ref{thm:pert_new_zeros} with $n = d$ shows that $F(z) = R_0(z) +
\frac{\eps}{z} - \conj{z}$ has $2d$ zeros located in two annuli around the
perturbation point $z_0=0$. 
Moreover, if $r>0$ and $\eps>0$ are sufficiently small, then there exist 
further $3d$ zeros corresponding to the $3d$ zeros of $R_0(z)-\conj{z}$ 
away from $z_0=0$.  The results in~\cite{LuceSeteLiesen2013} show that
the same perturbation behavior can be observed for Rhie's original
function $(1 - \eps) R_0(z) + \frac{\eps}{z} - \conj{z}$.  A numerical
example for this function with $d=7$ is shown in Figure~\ref{fig:circlens}.

\subsection{Generation of more than $2n$ zeros}
\label{sec:example_smallder}

\begin{figure}[t]
\subfigure[Unperturbed function]{%
\label{fig:smallder1}%
\includegraphics[width=.49\textwidth]{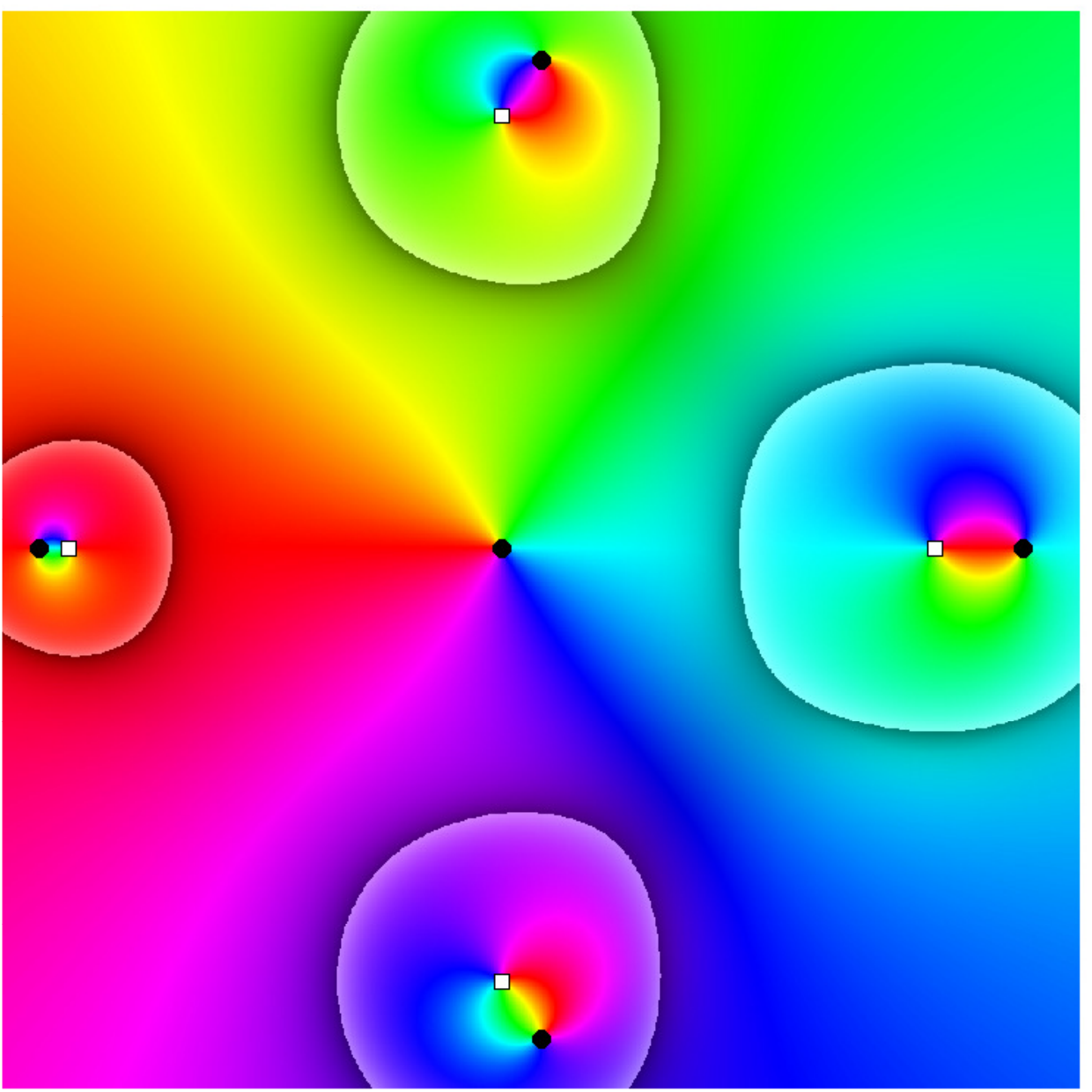}
}
\hfill
\subfigure[$\eps = 0.1$]{%
\label{fig:smallder2}%
\includegraphics[width=.49\textwidth]{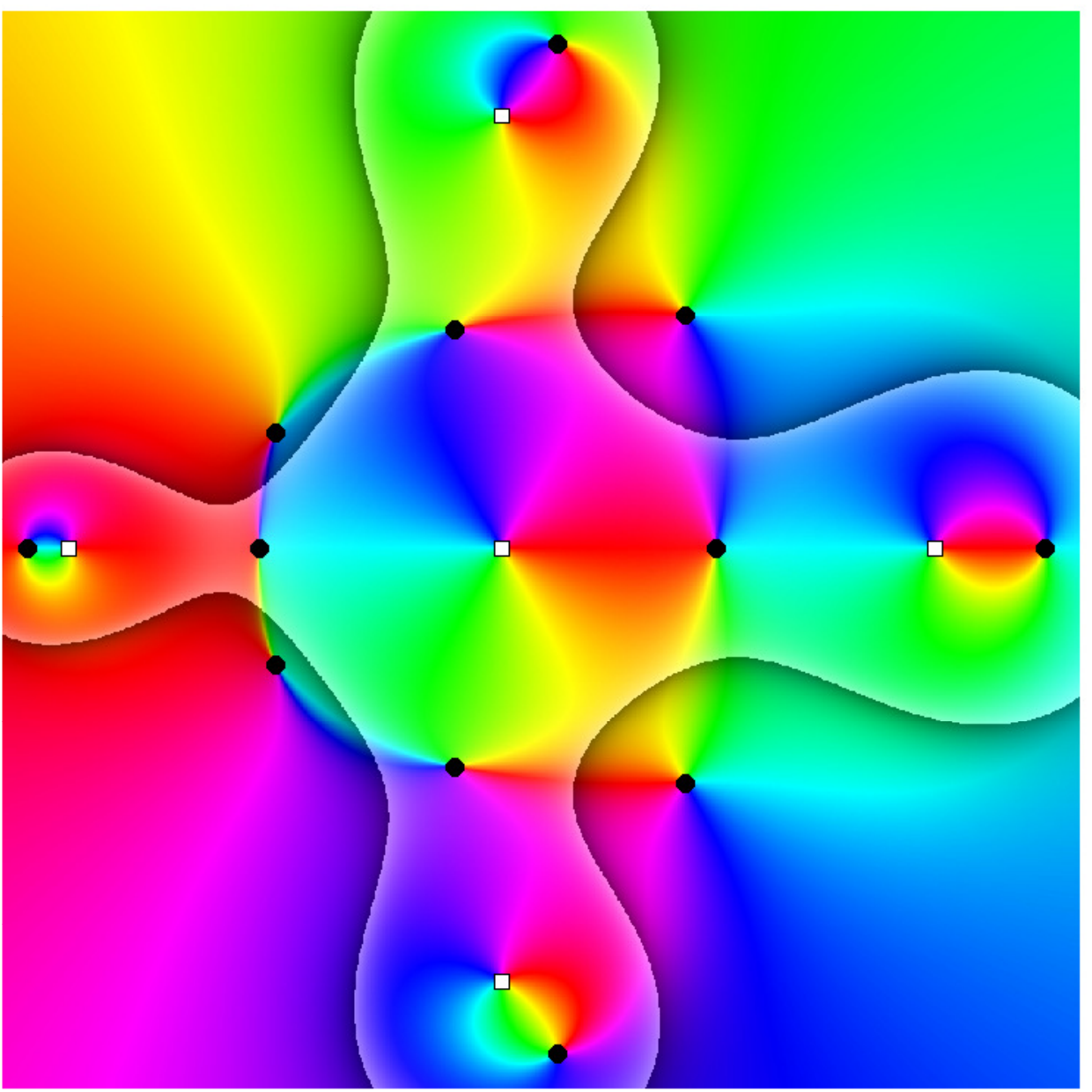}
}
\newline
\subfigure[$\eps = 0.05$]{%
\label{fig:smallder3}%
\includegraphics[width=.49\textwidth]{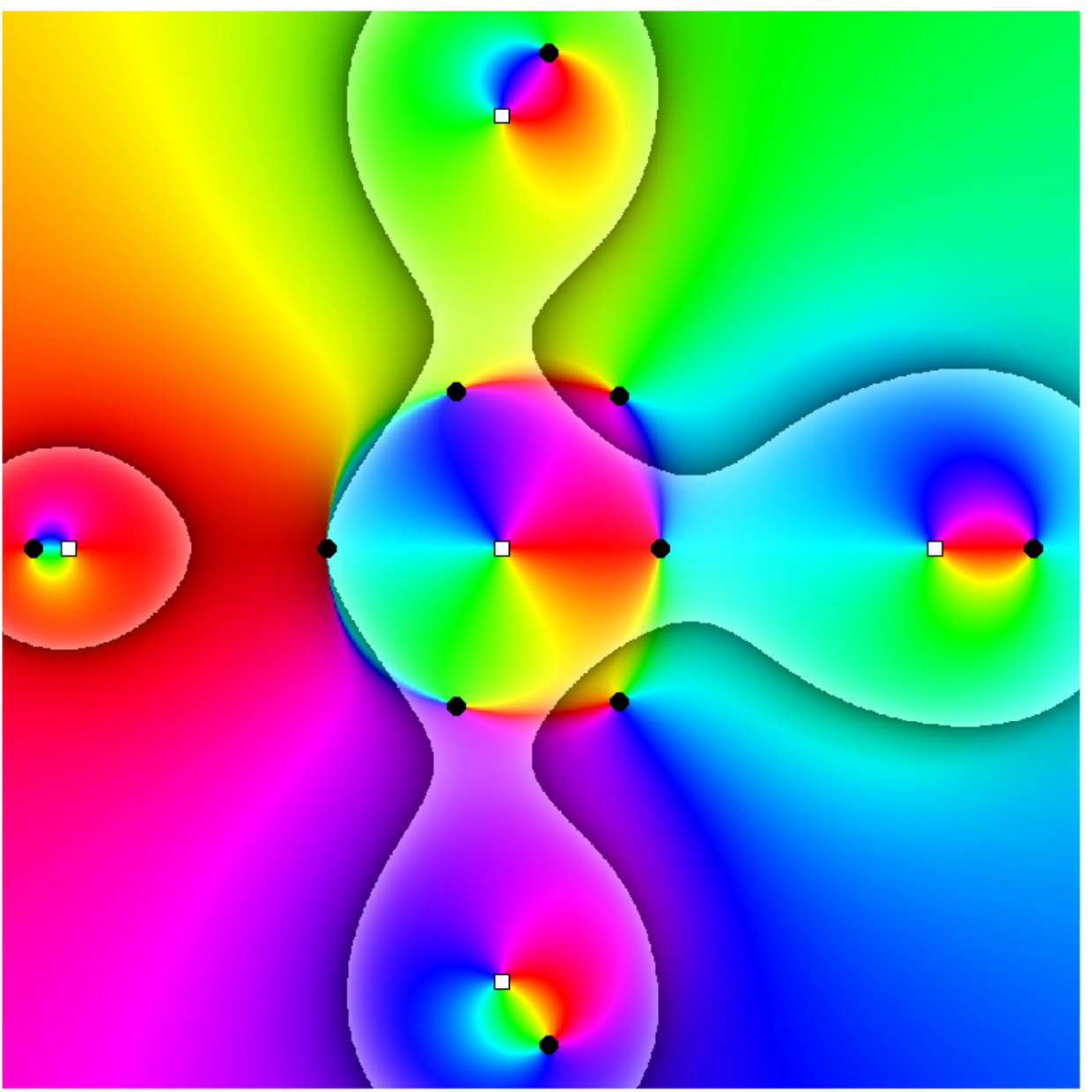}
}
\hfill
\subfigure[$\eps = 0.01$]{%
\label{fig:smallder4}%
\includegraphics[width=.49\textwidth]{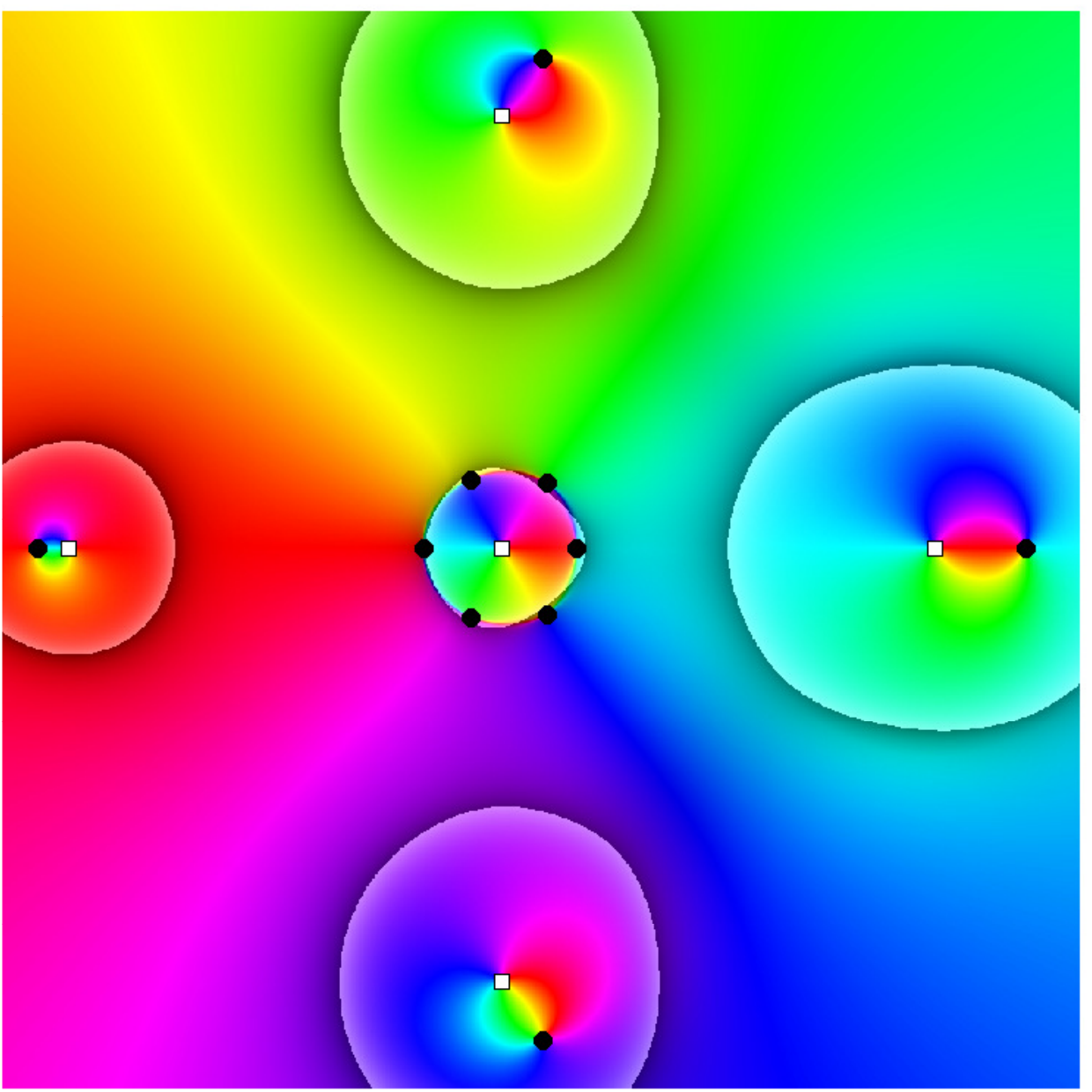}
}
\caption{More than $2n$ additional zeros may occur; see Section~\ref{sec:example_smallder}.}
\end{figure}

Theorem~\ref{thm:F_has_2n_zeros} \textit{(i)} gives the lower bound $2n$ on the number of
additional zeros due to the perturbation.  While we suspect that 
actually exactly $2n$ zeros appear for sufficiently small $\eps$ (see also 
Remark~\ref{rem:not_sharp}), we will now see that \emph{more} than
$2n$ zeros may appear if $\eps$ is not small enough.

Consider the function
\begin{equation*}
R(z) = \tfrac{\frac{1}{6} z^3 + \frac{1}{20} z^2}{z^4 - \frac{1}{10}},
\end{equation*}
for which $R(0)=R'(0)= 0$, $\abs{R''(0)} = 1$ and $\abs{R'''(0)} = 10$.  
A phase portrait of $R(z) - \conj{z}$ is shown in Figure~\ref{fig:smallder1}.

Figure~\ref{fig:smallder2} shows a phase portrait of
$R(z)+\tfrac{\eps}{z}-\conj{z}$ with $\eps=0.1$.
The perturbation results in 8 additional zeros
in the vicinity of $0$, although $n = 3$.
In this example the first non-vanishing derivative is dominated by a higher
derivative and hence our analysis from Section~\ref{sec:perturb} does
not apply, since $\eps$ is not small enough.

However, the second derivative becomes dominant as soon as $\eps$
is sufficiently small. Figure~\ref{fig:smallder3} shows the result of
perturbing with $\eps = 0.05$.  Now only $2n=6$ additional zeros occur,
as suggested by Theorem~\ref{thm:pert_new_zeros}.  If $\eps$ is
reduced further, as can be seen in Figure~\ref{fig:smallder4} for $\eps=0.01$,
the additional zeros approach those of the function $G(w)$ in 
Section~\ref{sec:perturb}, but the number of additional zeros (six) remains the 
same.

\subsection{Perturbing a sense-preserving zero}
\label{sec:example_preserve}

Figure~\ref{fig:pert_sense_pres} shows the result of perturbing a
sense-preserving zero by a pole.  The left plot shows a randomly
generated\footnote{More precisely, we chose $5$ points in the complex
plane, with real and imaginary parts drawn uniformly at random from
$[-1,1]$.  We then constructed $R(z)$ such that these random points
are the zeros of $R(z) - \conj{z}$.} rational harmonic function $f(z)
= R(z) - \conj{z}$, where $R(z)$ is of degree two (both the nominator and denominator
polynomials of $R(z)$ have degree two) and extremal.  The point $z_0$ is a
sense-preserving zero of $f(z)$ with $\abs{R'(z_0)} \approx 1.176 >
1$.  We add a pole at $z_0$ with $\eps=0.025$; the resulting
function is shown in Figure~\ref{fig:pert_sense_pres} (right).  As
suggested by Theorem~\ref{thm:other_z0}, two sense-preserving zeros appear.

\begin{figure}[t]
\includegraphics[width=.49\textwidth]{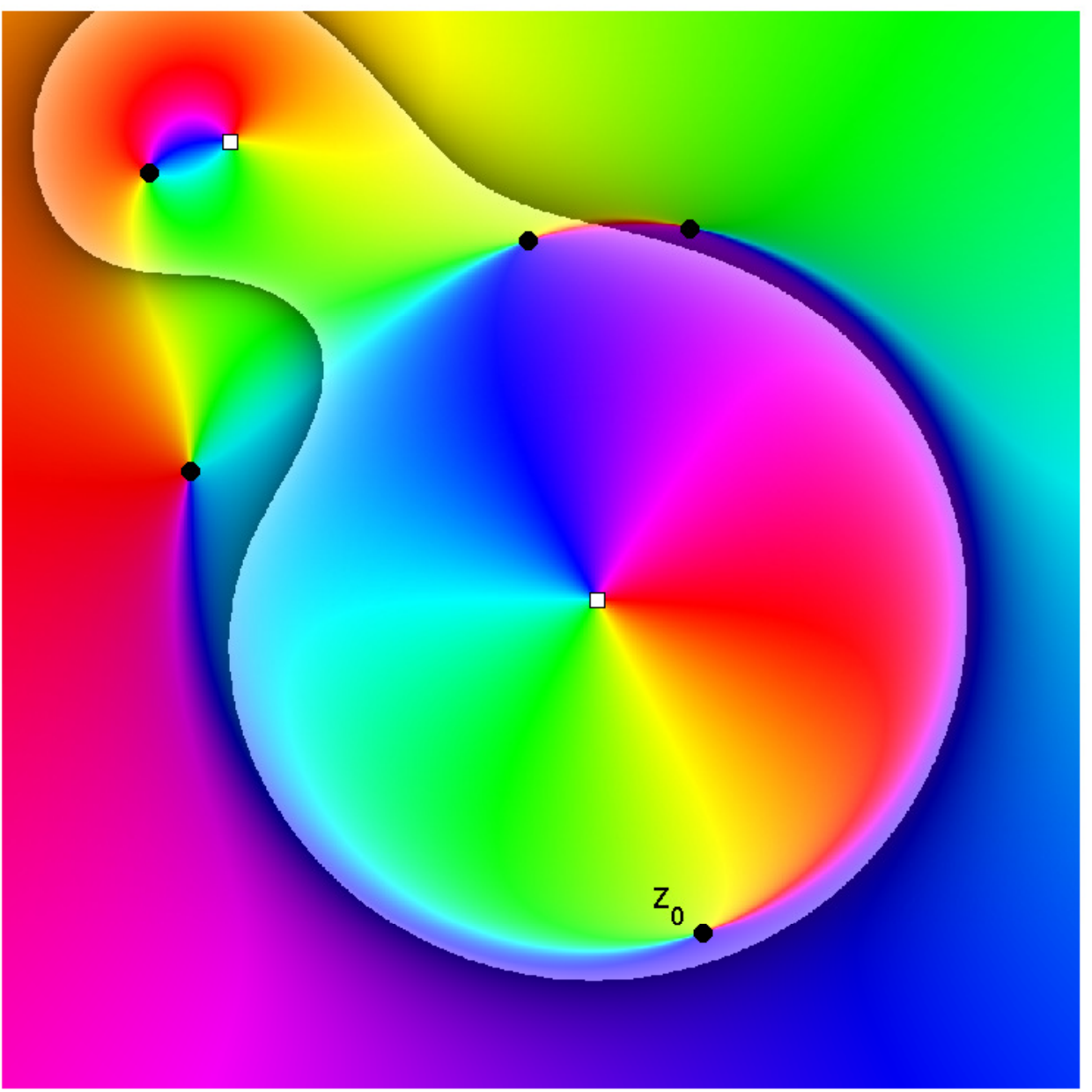}
\hfill
\includegraphics[width=.49\textwidth]{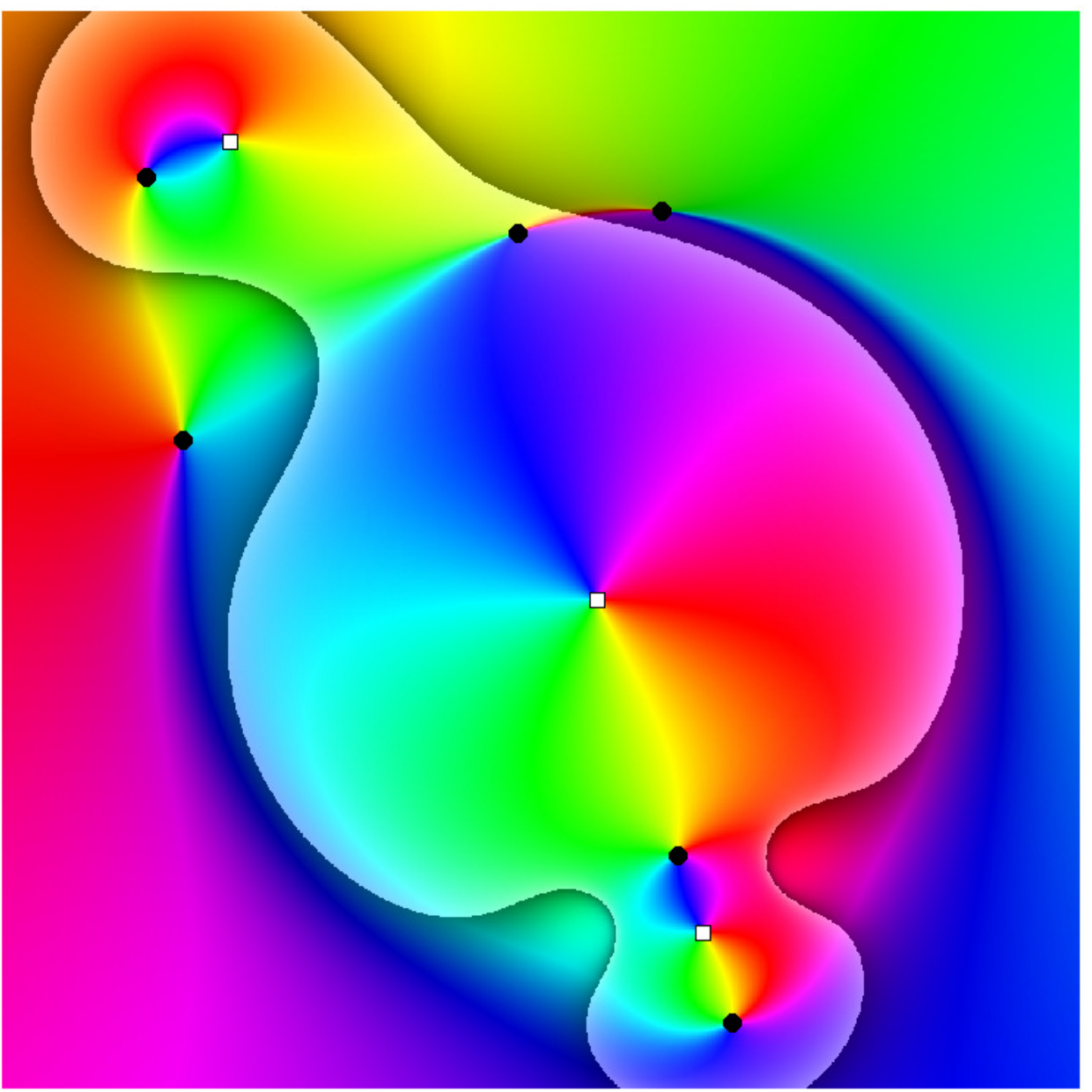}
\caption{Perturbation at a sense-preserving zero; see
Section~\ref{sec:example_preserve}.}
\label{fig:pert_sense_pres}
\end{figure}

\subsection{Iterative perturbation}
\label{sec:example_randpert}

Starting from the randomly chosen extremal rational function 
$R(z)$ of degree two introduced in the previous section, 
our goal is now to successively add two poles at sense-reversing 
zeros such that extremality is maintained by each perturbation. 
We will thus obtain extremal rational functions of degrees three and four.  
This procedure is displayed in Figure~\ref{fig:randpert}, which we will 
discuss in detail now.

\begin{figure}[p]
\subfigure[]{%
\label{fig:randpert1}%
\includegraphics[width=.45\textwidth]{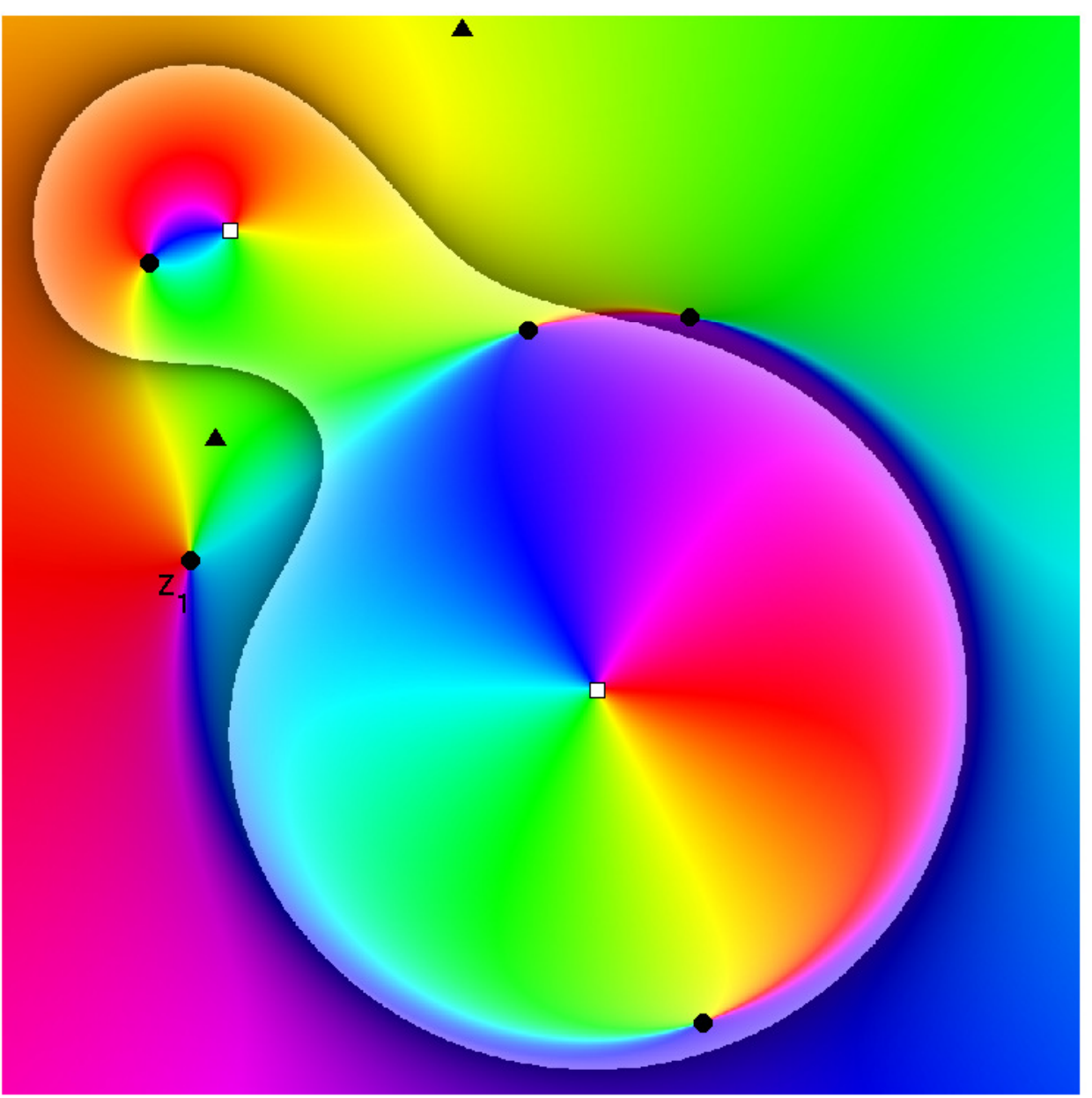}
}
\hfill
\subfigure[]{%
\label{fig:randpert2}%
\includegraphics[width=.45\textwidth]{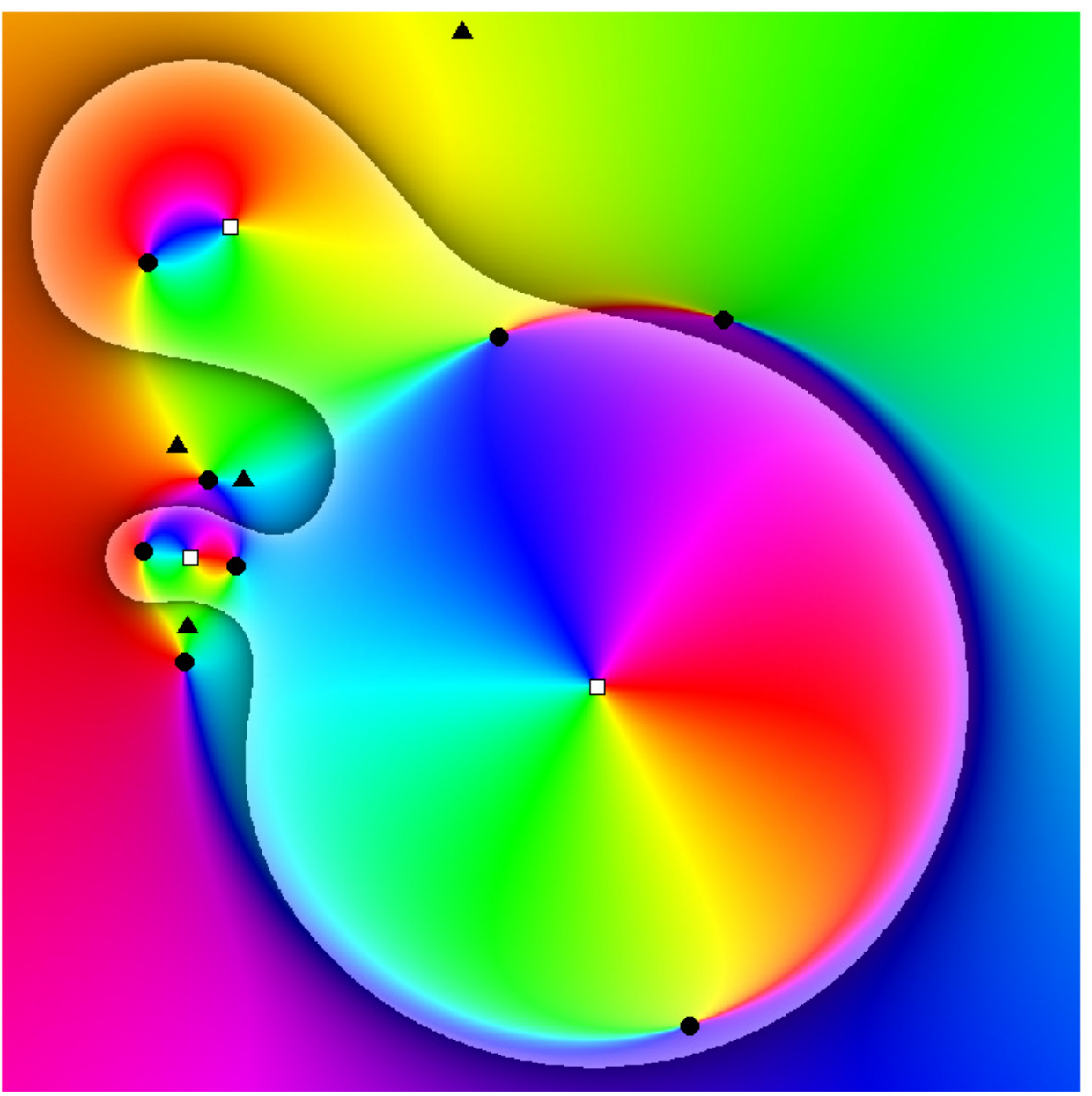}
}
\newline
\subfigure[]{%
\label{fig:randpert3}%
\includegraphics[width=.45\textwidth]{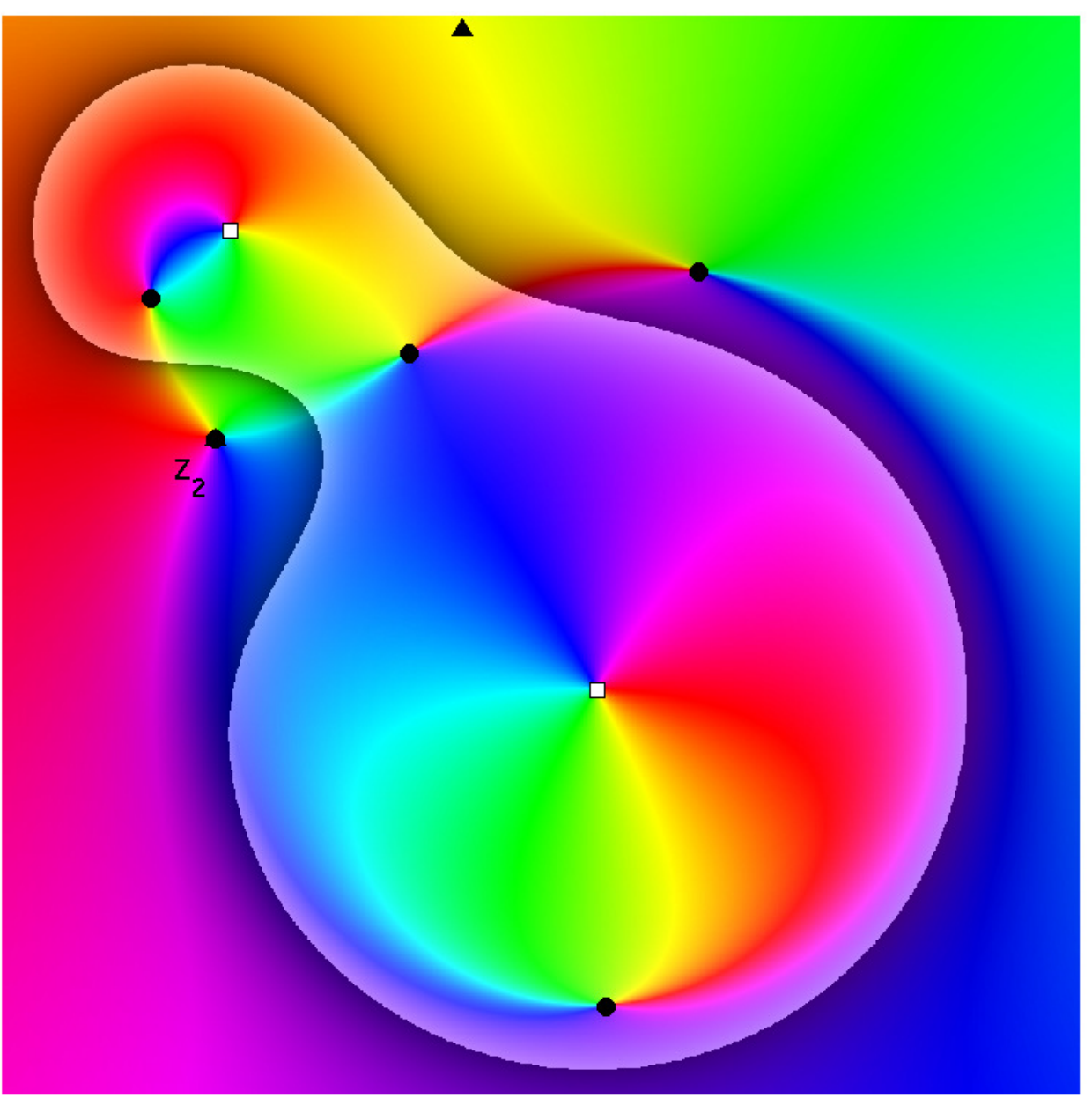}
}
\hfill
\subfigure[]{%
\label{fig:randpert4}%
\includegraphics[width=.45\textwidth]{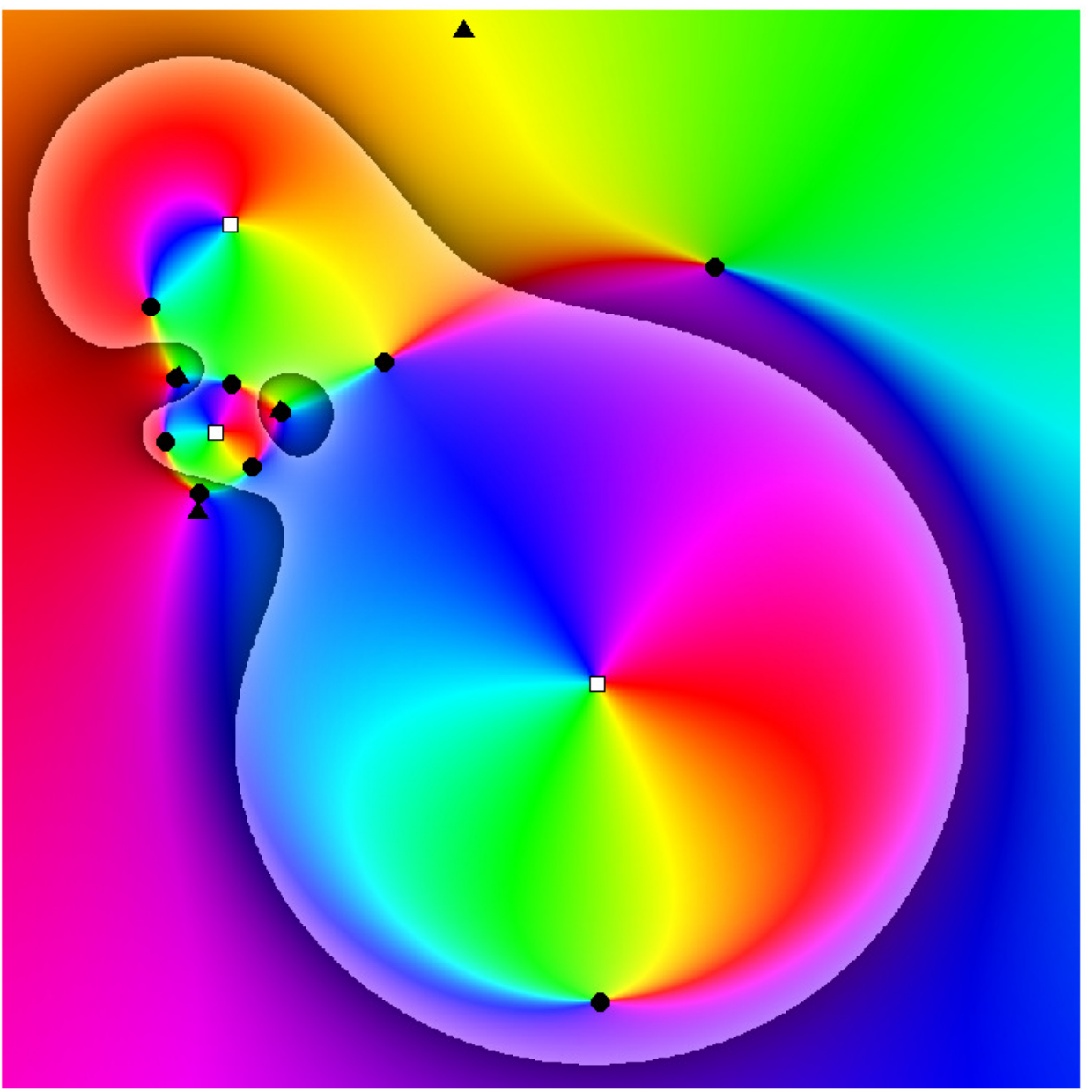}
}
\newline
\subfigure[]{%
\label{fig:randpert5}%
\includegraphics[width=.45\textwidth]{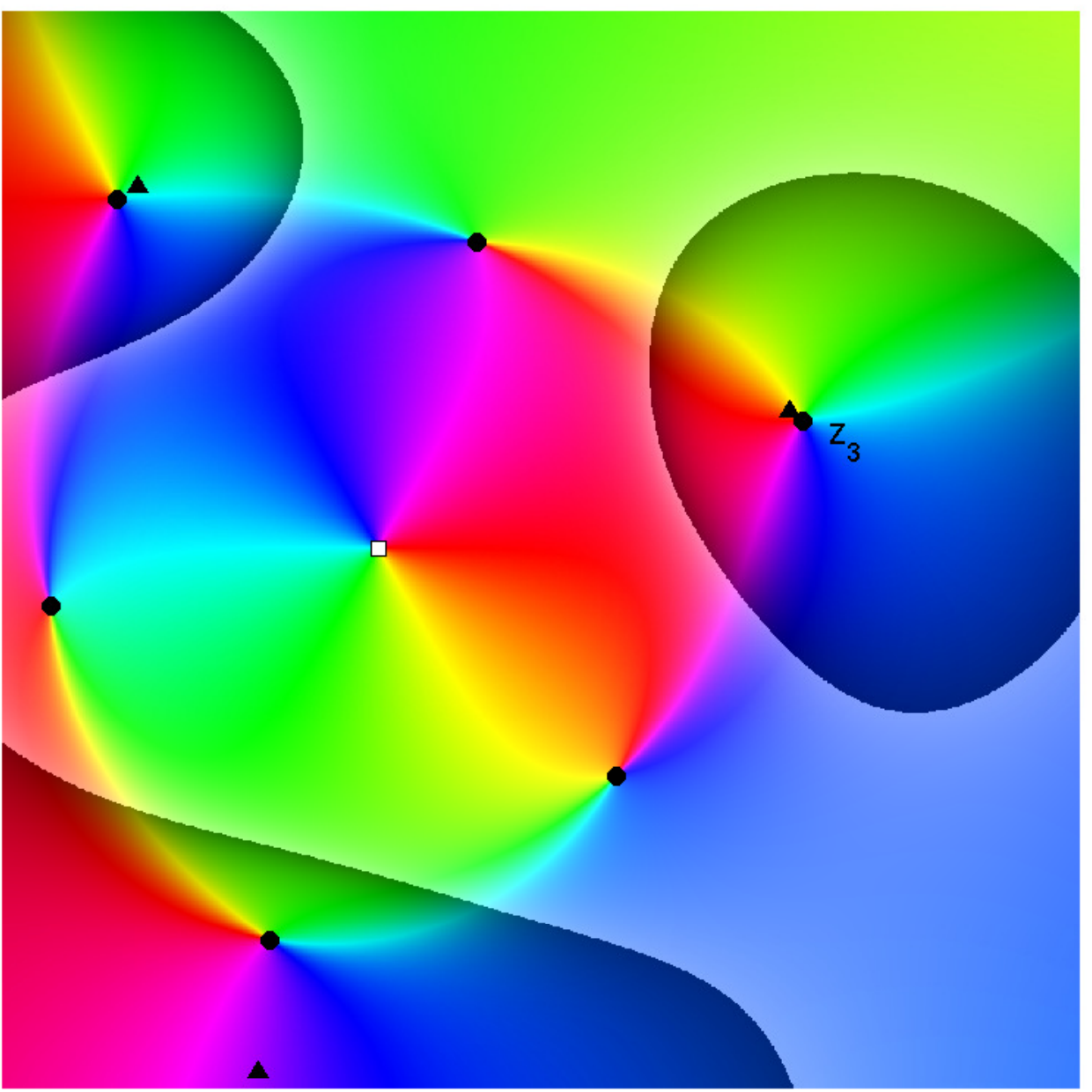}
}
\hfill
\subfigure[]{%
\label{fig:randpert6}%
\includegraphics[width=.45\textwidth]{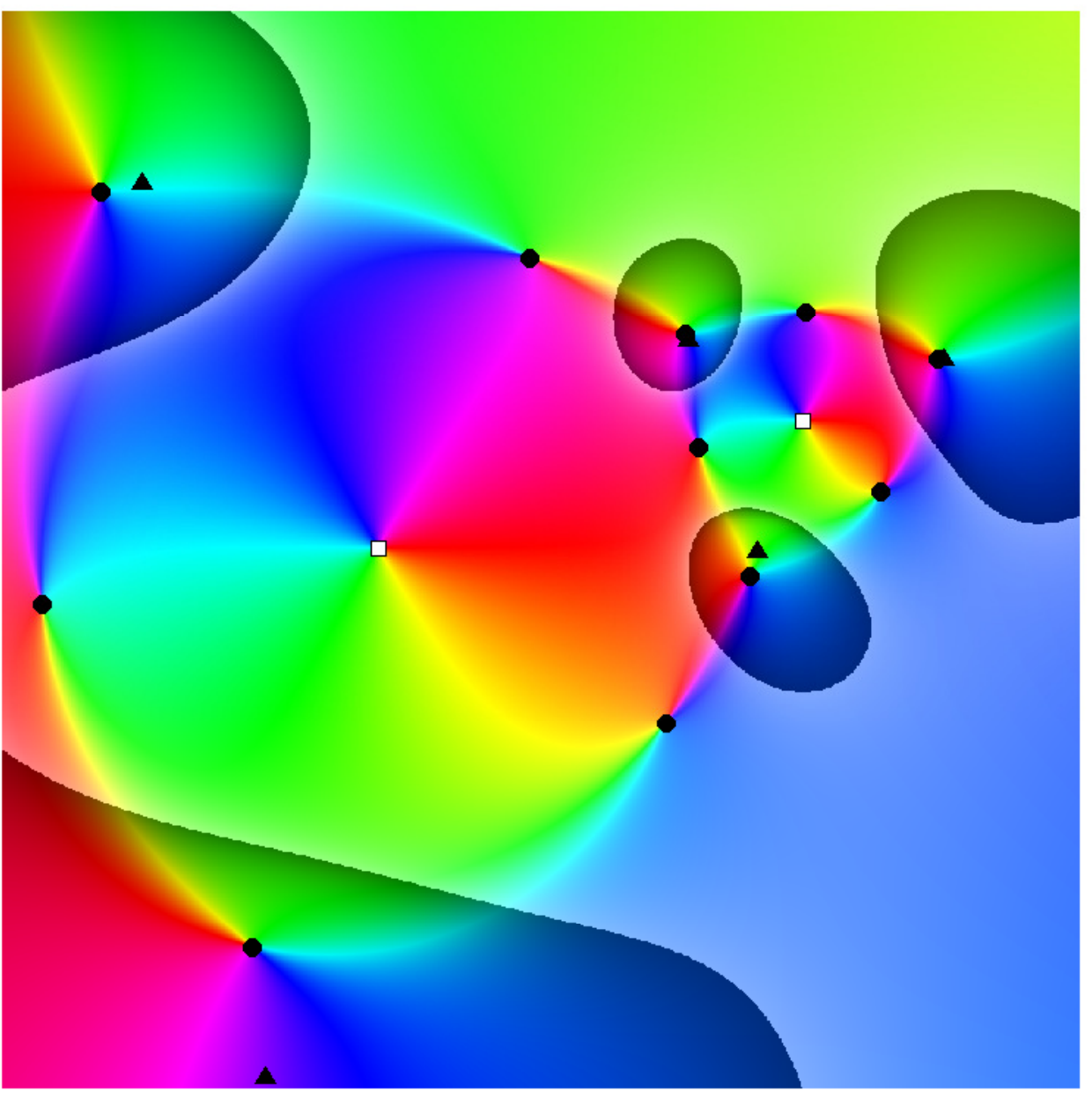}
}
\caption{Iterative perturbation of a random rational function; see
Section~\ref{sec:example_randpert}.  Black triangles indicate the zeros of
$R'(z)$.
\label{fig:randpert}}
\end{figure}

The initial situation with $R_1(z)=R(z)$ is depicted in 
Figure~\ref{fig:randpert1}, where we plot the phase portrait 
of $R_1(z) - \conj{z}$.  Black triangles show zeros of $R_1'(z)$ 
and white squares show the poles of $R_1(z)$.  The picture shows 
that $R_1(z)-\conj{z}$ has three sense-preserving zeros (black disks in 
the bright region), and two sense-reversing zeros (black disks in the dark 
region).

Clearly, neither of the two sense-reversing zeros satisfies the
conditions of Theorem~\ref{thm:pert_new_zeros}.  Denote by $z_1$ the
leftmost sense-reversing zero of $R_1(z) - \conj{z}$; we
have $\abs{R_1'(z_1)} \approx 0.5572$.  Numerical experiments
show that perturbing $R_1(z)-\conj{z}$ with the pole 
$\frac{\eps_1}{z - z_1}$ results for all $\eps_1 > 0$ small enough in
a function that has four additional zeros nearby $z_1$. The result of 
such a perturbation is shown in Figure~\ref{fig:randpert2}.
The four additional zeros are explained by Theorem~\ref{thm:other_z0}, 
and reducing $\eps_1$ further does not result in further zeros
in our numerical experiments.

Figure~\ref{fig:randpert3} shows the phase portrait of the function
$R_2(z) - \conj{z}$, where $R_2(z) = R_1(z) + c$ with a deliberately
chosen constant $c \in \C$ such that $R_2(z) - \conj{z}$ has a
sense-reversing zero $z_2$ that coincides (numerically) with a zero of
$R_2'(z) = R_1'(z)$ (leftmost zero of $R_2'(z)$).

The effect of adding a pole at $z_2$ to $R_2(z)$, i.e., forming
$R_3(z) = R_2(z) + \frac{\eps_2}{z-z_2}$, is shown in the phase
portrait of $R_3(z) - \conj{z}$ in Figure~\ref{fig:randpert4}.  An
enlarged view on the perturbed region is given in
Figure~\ref{fig:randpert5}.  We see that $R_3(z) - \conj{z}$ has $5$
more zeros than $R_2(z) - \conj{z}$.  Here $\eps_2 = 4.5 \cdot
10^{-3}$, and we have chosen this value because a newly created zero
of $R_3(z) - \conj{z}$ lies in the immediate vicinity of a zero of
$R_3'(z)$ (rightmost zero in Figure~\ref{fig:randpert5}).

In order to obtain conditions that allow another application of
Theorem~\ref{thm:pert_new_zeros}, we could have tried to add another
constant to $R_3(z)$ such that the conditions of the theorem are met.
However, the rightmost zero of $R_3(z) - \conj{z}$, say, $z_3$, is
already very close to a zero of $R_3'(z)$.  Figure~\ref{fig:randpert6}
shows the effect of adding a pole at this position, i.e. setting
$R_4(z) = R_3(z) + \frac{\eps_3}{z - z_3}$ and considering the
function $R_4(z) - \conj{z}$.  Although the conditions of the theorem
are not exactly met (we have $\abs{R_3'(z_3)} \approx 0.7936$), the same effect
as in
Figures~\ref{fig:randpert3}--\ref{fig:randpert4} can be observed.  In
particular, $R_4(z)$ is again extremal. 

Notice also that it has been possible to choose the value $\eps_3$
such that two newly created zeros of $R_4(z) - \conj{z}$ are very
close to newly created zeros of $R_4'(z)$, suggesting that it may be
possible to repeat this perturbation procedure at least one more time.

\subsection{Extension to complex residues}
\label{sec:complex_res}

Our motivation for studying the effect of adding poles to $R(z) -
\conj{z}$ comes from an astrophysical application in gravitational
microlensing.  This astrophysical application requires adding simple
poles with \emph{positive} residues.  From a mathematical point of
view, one may wonder about the effect of adding poles with
\emph{complex} residues on the number of (newly created) zeros.  We
will now give a brief and informal discussion of this case.

\begin{figure}[t]
\subfigure[$\theta = 0$ (positive perturbation)]{%
\label{fig:theta=0}%
\includegraphics[width=.49\textwidth]{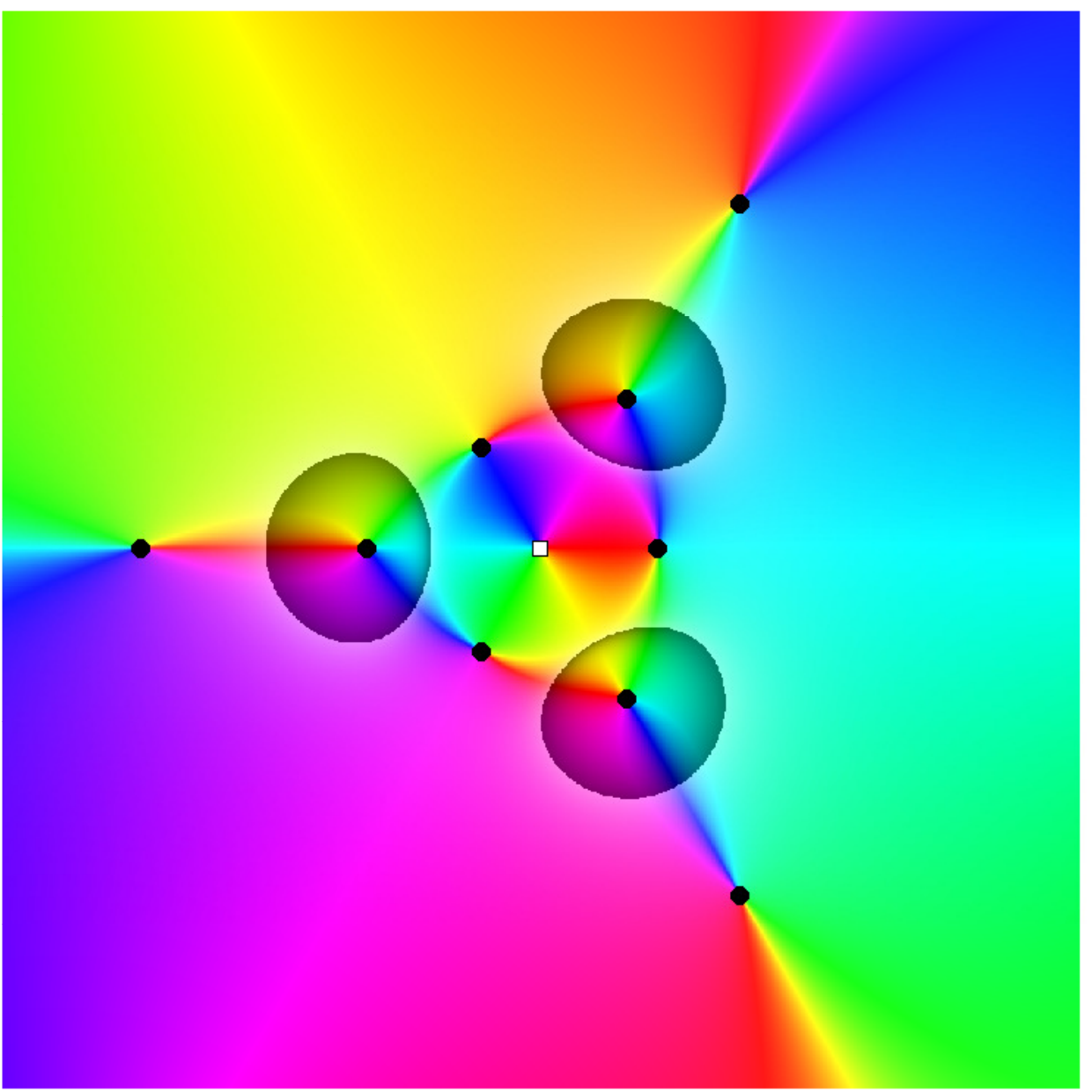}
}
\hfill
\subfigure[$\theta = 0.1 \pi$]{%
\label{fig:theta=0.1}%
\includegraphics[width=.49\textwidth]{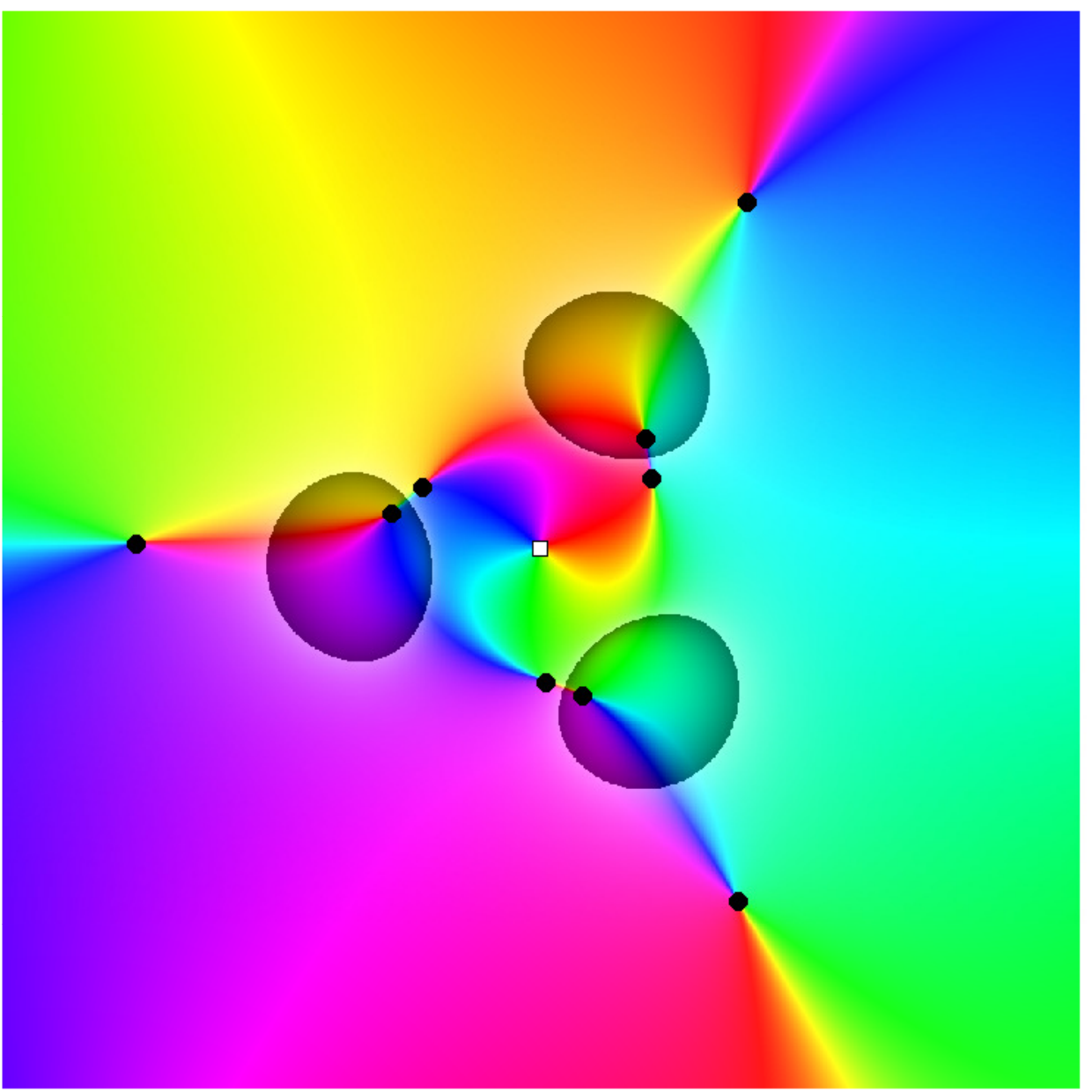}
}
\newline
\subfigure[$\theta = 0.2 \pi$]{%
\label{fig:theta=0.2}%
\includegraphics[width=.49\textwidth]{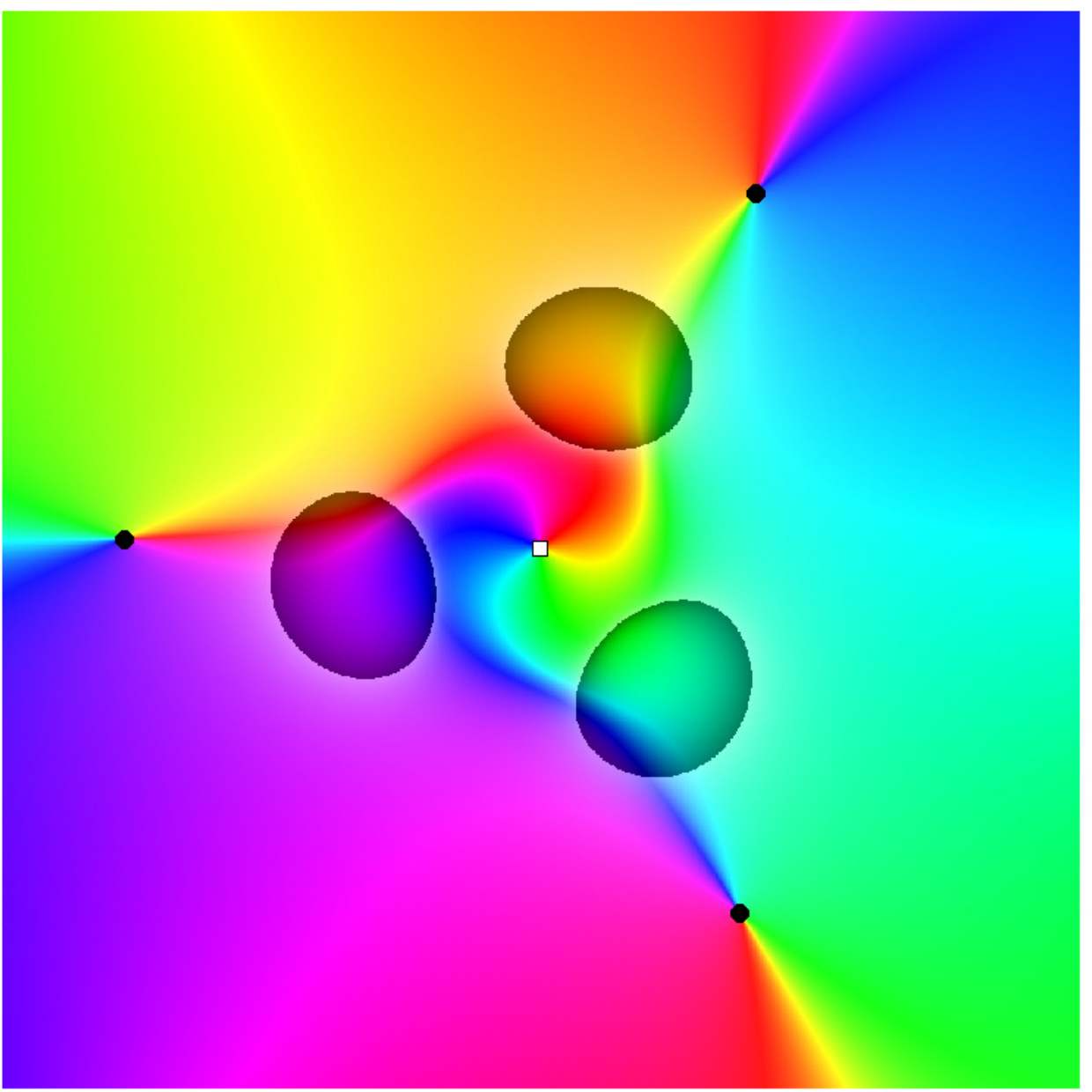}
}
\hfill
\subfigure[$\theta = \pi$ (negative perturbation)]{%
\label{fig:theta=1}%
\includegraphics[width=.49\textwidth]{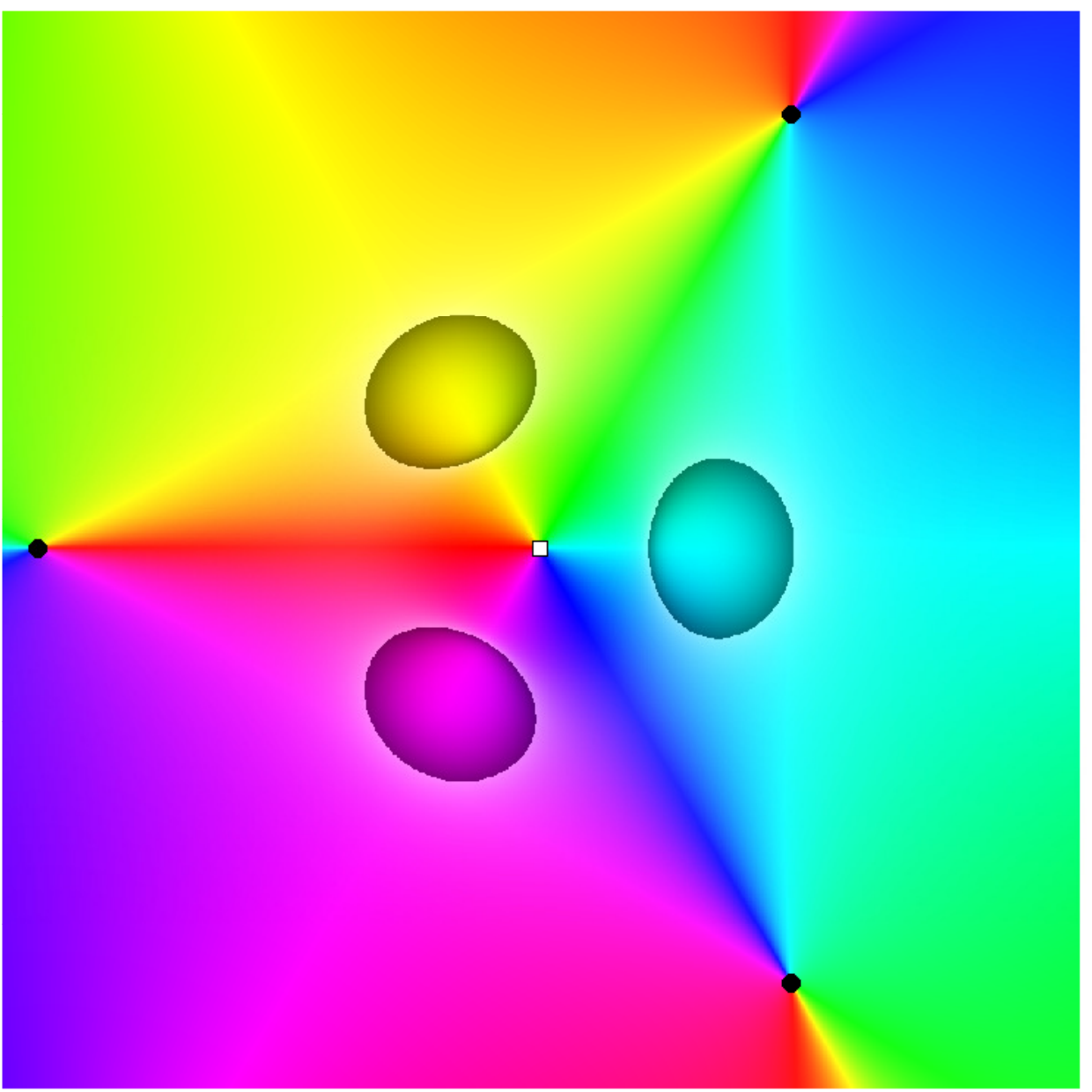}
}
\caption{Phase portraits of Rhie's function $R_\eps(z)$ from~\eqref{eqn:R_eps} 
with $\eps > 0$ replaced by $\eps e^{i \theta}$; see 
Section~\ref{sec:complex_res}.}
\label{fig:complex_res}
\end{figure}

Consider Rhie's function from~\eqref{eqn:R_eps}, where we now replace
$\eps$ by $\eps e^{i \theta}$, with $\eps > 0$ and $\theta \in \R$.
Figure~\ref{fig:complex_res} shows phase portraits of $R_{\eps e^{i
\theta}}(z)$ centered at the origin for $n = 3$ and several values of
$\theta$.  For positive perturbation ($\theta = 0$) the function
$R_\eps(z) - \conj{z}$ has $2n = 6$ zeros close to the origin
(Figure~\ref{fig:theta=0}).  We observe that with increasing argument,
the zeros close to the origin approach each other
(Figure~\ref{fig:theta=0.1}) and finally vanish
(Figures~\ref{fig:theta=0.2} and~\ref{fig:theta=1}).

Heuristically this effect can be explained as follows.  Let $z_0$ be a
zero of $f(z) = R(z) - \conj{z}$, where $R(z)$ is rational with
$\deg(R) \geq 2$.  Consider the perturbed function $F(z) = f(z) +
\tfrac{\eps}{z-z_0}$.  This function can be approximated (after
translation $w \coloneq z-z_0$) by the function $G(w)$
in~\eqref{eqn:defn_G}.  The zeros of $G(w)$ close to $z_0$ give rise
to zeros of $F(z)$; see Sections~\ref{sec:near_z0}
and~\ref{sec:other_z0}.  The zeros of $G(w)$ are the solutions of the
equation
\begin{equation*}
c w^n + \eps - \abs{w}^2 = 0.
\end{equation*}
For small values of $w$, this equation is \emph{approximated} by $\eps - 
\abs{w}^2 = 0$, which has solutions only for $\eps \geq 0$.  This suggests that 
$F(z)$ has no zeros near $z_0$ for sufficiently small $\eps$ chosen ``away from 
the positive real axis''.

\subsection{Extension to poles of higher order}
\label{sec:hop}

So far we considered perturbations by simple poles.  One may wonder about the 
effects of adding poles of higher order.  We give a brief and informal 
discussion of this case.

For simplicity, let $f(z)$ be as in Theorem~\ref{thm:other_z0}.  Suppose
$z_0 \in \C$ is not a pole of $f(z)$ and consider the function $F(z) =
f(z) + \tfrac{\eps}{ (z-z_0)^k }$, where now $k \geq 1$, and assume
that $F(z)$ is regular as well.  As before, we can show that on a
sufficiently large circle $\Gamma$ we have $V(f; \Gamma) = V(F;
\Gamma)$.  Further, there exists a circle $\gamma$ centered at $z_0$,
not containing any zero or pole of $f(z)$ (except possibly $z_0$),
such that for sufficiently small $\eps$ the functions $f(z)$ and
$F(z)$ have the same number of zeros with same indices outside
$\gamma$, and that $V(f; \gamma) = V(F; \gamma)$ holds.  Let us denote
by $n_+^{\new}$ and $n_-^{\new}$ the number of sense-preserving and
sense-reversing zeros of $F(z)$ inside $\gamma$, respectively.  By the
argument principle (Theorem~\ref{thm:arg_principle}) we then have
\begin{equation*}
    \ind(z_0; f) = V(f; \gamma) = V(F; \gamma) = -k + n_+^{\new} - n_-^{\new},
\end{equation*}
or, equivalently,
$n_+^{\new} - n_-^{\new} = k + \ind(z_0; f)$.  Thus the total number
of zeros of the regular function $F(z)$ inside $\gamma$ is
\begin{equation}
\label{eqn:ho_bound}
    n_+^{\new} + n_-^{\new} = k + \ind(z_0; f) + 2 n_-^{\new}
        \geq k + \ind(z_0; f).
\end{equation}

\begin{figure}
\subfigure[Unperturbed function.]{%
\label{fig:hop_initial}%
\includegraphics[width=.49\textwidth]{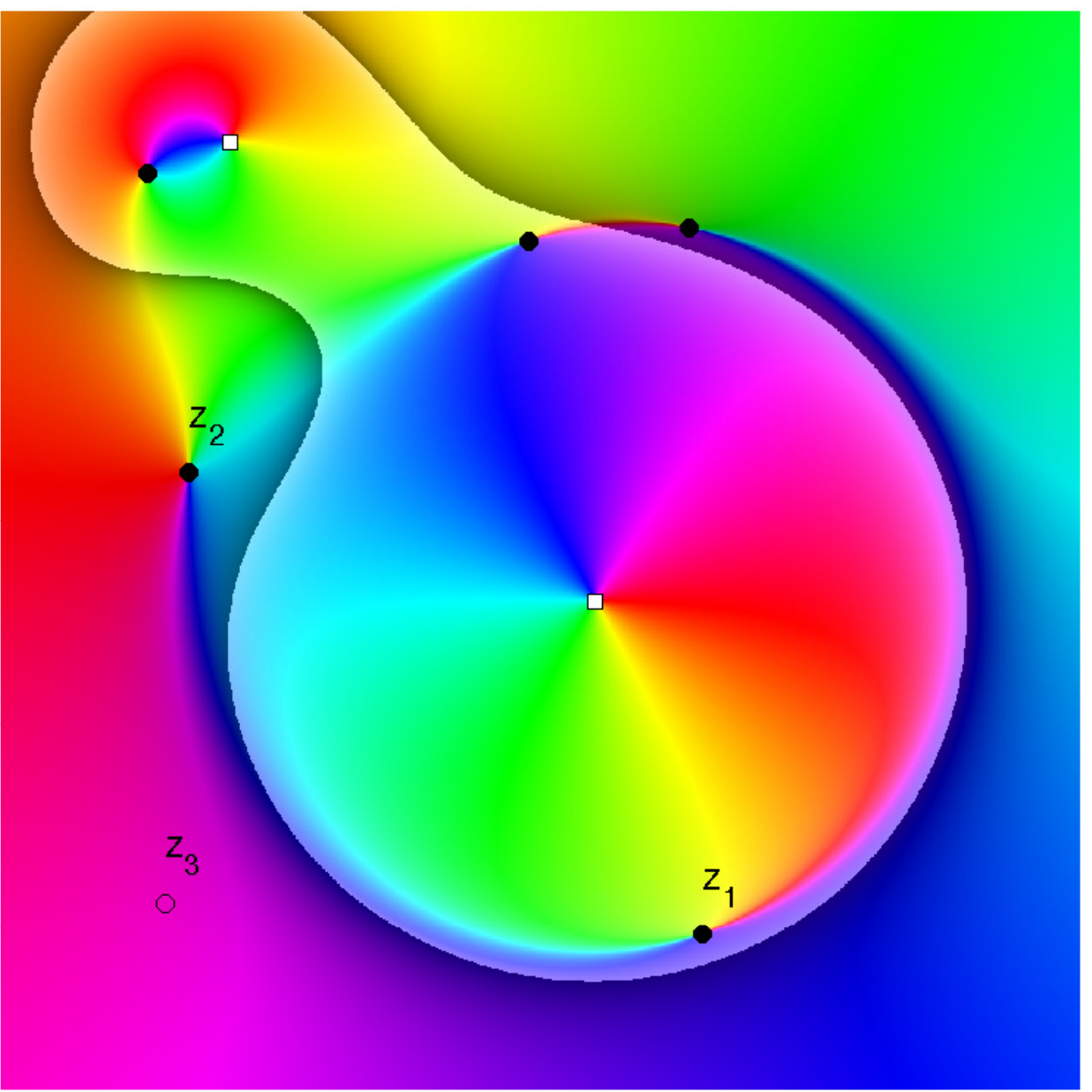}
}
\hfill
\subfigure[Perturbing a sense-preserving zero.]{%
\label{fig:hop_at_sp_zero}%
\includegraphics[width=.49\textwidth]{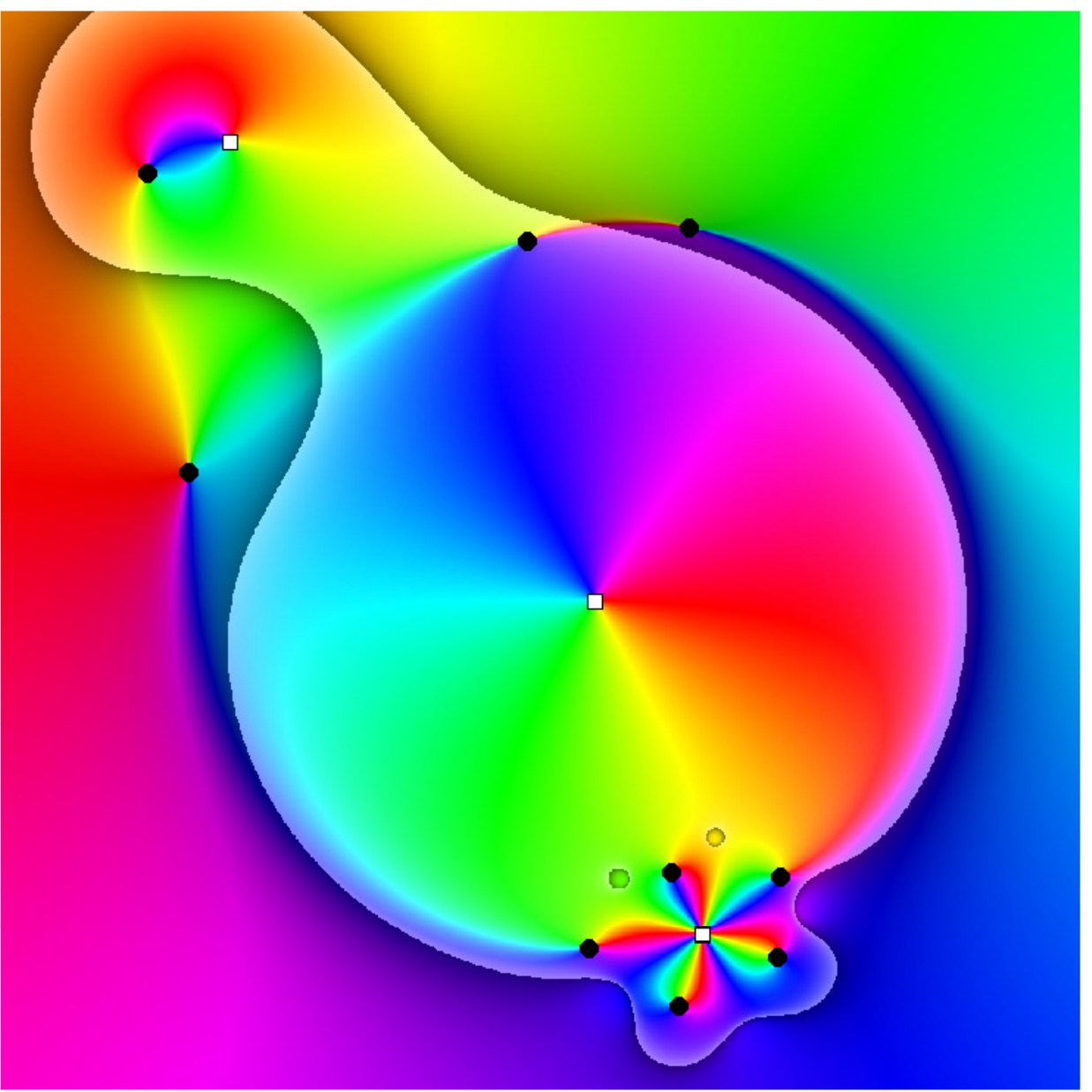}
}
\newline
\subfigure[Perturbing a sense-reversing zero.]{%
\includegraphics[width=.49\textwidth]{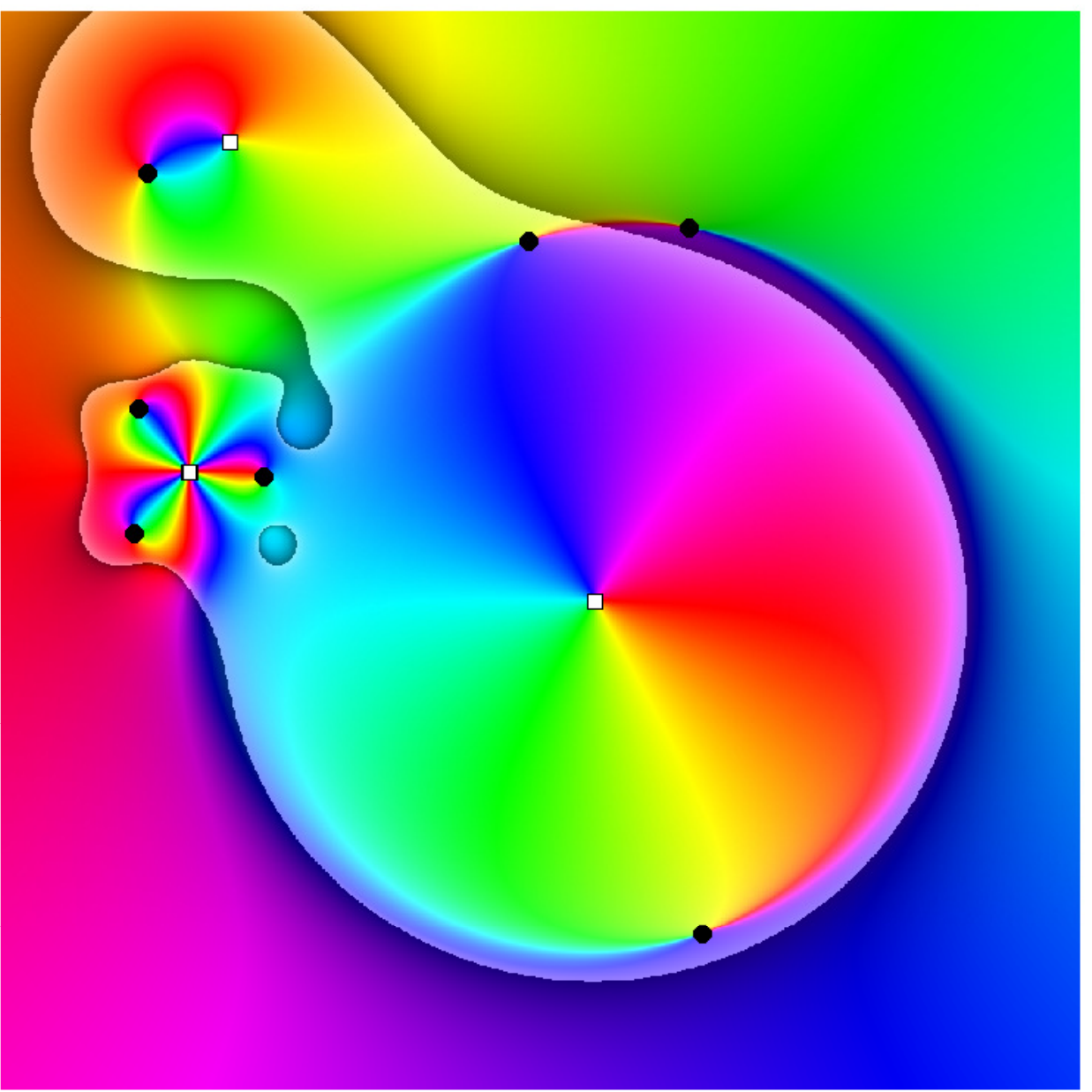}
}
\hfill
\subfigure[Perturbing a non-zero.]{%
\label{fig:hop_at_no_zero}%
\includegraphics[width=.49\textwidth]{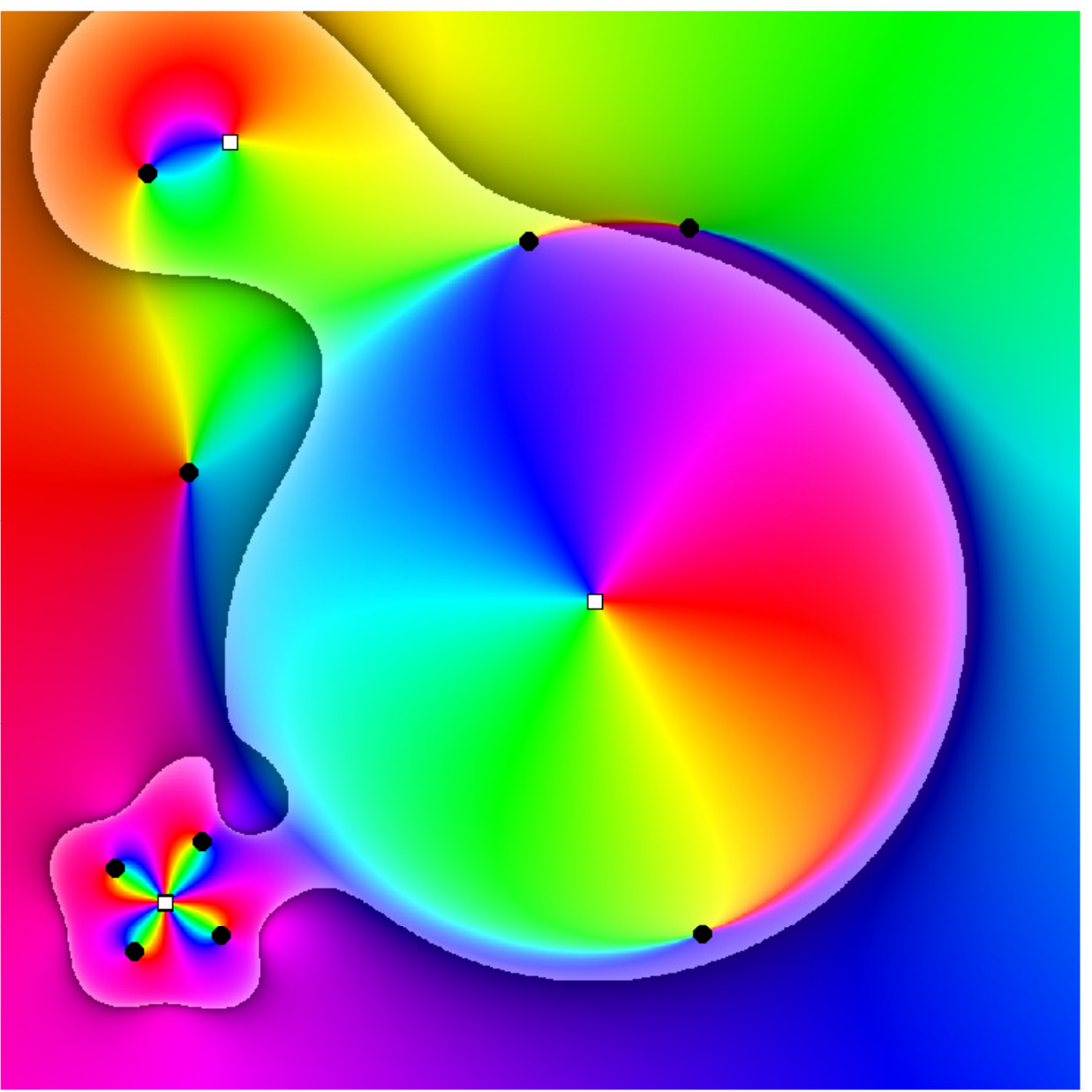}
}
\caption{Perturbations with poles of higher order; see 
Section~\ref{sec:hop}.\label{fig:hop}}
\end{figure}

We illustrate this result in Figure~\ref{fig:hop}.  In
Figure~\ref{fig:hop_initial}, the same function as used in
Section~\ref{sec:example_preserve} is shown; the indicated points
correspond to a sense-preserving zero ($\ind(z_1; f) = 1$), a
sense-reversing zero ($\ind(z_2; f)=-1)$ and a non-exceptional point
($\ind(z_3; f) = 0)$. 

Adding a pole of order $k=4$ with residue $\eps = 10^{-5}$ at the
designated positions $z_1, z_2$ and $z_3$ results in $5$, $3$, and $4$
new solutions nearby the position of the pole, which is shown in
Figures~\ref{fig:hop_at_sp_zero}--\ref{fig:hop_at_no_zero}.  In each
case, the minimum number of zeros are created (see the
bound~\eqref{eqn:ho_bound}), and $n_-^{\new} = 0$.

We performed extensive numerical experiments for various combinations
of $k$, index of $z_0$ and the number of vanishing derivatives at
$z_0$.  In each case, if $\eps$ was chosen sufficiently small, the
minimum number of zeros was created and $n_-^{\new}=0$.  Further, the
effect of complex residues $\eps$ we observed was significantly
different from the situation in Section~\ref{sec:complex_res}; the
created zeros are merely rotated according to the value of
$\arg(\eps)$.


\section{Conclusions and Outlook} \label{sec:conclusions}

We have generalized Rhie's technique~\cite{Rhie2003} for
constructing extremal point lenses to arbitrary rational
harmonic functions of the form $f(z) = R(z) - \conj{z}$.  We have
shown that if $f(z_0) = 0$ and the first $n-2$ derivatives of $R(z)$
vanish at $z_0$, while $R^{(n-1)}(z_0)\neq 0$,
then the function $f(z) + \frac{\eps}{z - z_0}$ has
at least $2n$ zeros near the circle $\abs{z - z_0} = \sqrt{\eps}$,
provided $\eps > 0$ is sufficiently small.  In particular, our general results 
cover Rhie's original construction; see 
Corollary~\ref{cor:produce_extremal_function}.
We also gave lower bounds on the number of additional zeros if $f(z)$ is
perturbed by a pole at any other point $z_0 \in \C$.

We have briefly studied extensions of our results to complex residues
(Section~\ref{sec:complex_res}) and poles of higher order
(Section~\ref{sec:hop}) by means of examples, but we did not pursue
these directions with full rigor.

Theorems~\ref{thm:pert_new_zeros} and~\ref{thm:other_z0} give
\emph{sufficient conditions} for \emph{lower bounds} on the number of
newly created zeros in the vicinity of the point where the pole is
added, which raises two open questions.  Firstly, the examples in
Section~\ref{sec:example_smallder} demonstrate that the stated
conditions are not \emph{necessary} for the creation of the stated
number of zeros.  They show that more images can be created if the
perturbation residue $\eps$ is chosen somewhat ``too large''.  It would be
interesting to quantify this effect.  Secondly, we believe that for
sufficiently small perturbations, no more than the stated number of
zeros are created; see also Remark~\ref{rem:not_sharp}.  We have
performed extensive numerical experiments that support this claim.
However, a rigorous analysis is yet to be done.

\paragraph*{Acknowledgements}
We would like to thank Elias Wegert for comments on the manuscript and
creating the color scheme used in the illustrations.  We are grateful
to the anonymous referee for carefully reading the manuscript and for
giving us many useful suggestions.

The work of R. Luce was supported by Deutsche Forschungsgemeinschaft,
Cluster of Excellence ``UniCat''.

\bibliographystyle{plain}
\bibliography{perturb}

\end{document}